\documentclass[a4paper, 11pt]{amsart}
\usepackage{epsfig}
\usepackage{amsthm,amsfonts}
\usepackage{amssymb,graphicx,color}
\usepackage[all]{xy}
\usepackage{cancel} 
\usepackage{verbatim}
\usepackage{hyperref}
\usepackage{enumitem}
\usepackage{booktabs}

\usepackage{biblatex}

\newtheorem{theorem}{Theorem}[section]
\newtheorem*{theorem*}{Theorem}
\newtheorem{lemma}[theorem]{Lemma}
\newtheorem{corollary}[theorem]{Corollary}
\newtheorem{proposition}[theorem]{Proposition}

\newtheorem{problem}{Problem}

\theoremstyle{definition}
\newtheorem{definition}[theorem]{Definition}
\newtheorem{example}[theorem]{Example}

\theoremstyle{remark}
\newtheorem{remark}[theorem]{Remark}

\newcommand{\R}{\mathbb{R}}
\newcommand{\C}{\mathbb{C}}
\newcommand{\Z}{\mathbb{Z}}

\newcommand{\N}{\mathbb{N}}
\renewcommand{\Re}{\mathfrak{Re}}
\renewcommand{\Im}{\mathfrak{Im}}
\newcommand{\Cont}{\text{Cont}}

\title{Outer Lipschitz geometry of unions of tangential nodal sets}

\author[J. E. Sampaio]{José Edson Sampaio}

\author[L. M. Sousa]{Luan Moura Sousa}

\author[E. F. de Souza]{Emanoel Ferreira de Souza}

\begin{document}

\begin{abstract}
In 2000, Birbrair and Fernandes proved that the Hölder Semicomplex of the branches of a germ of a curve with its contact orders is a complete invariant of outer bi-Lipschitz equivalence of definable sets in polynomially bounded o-minimal structures. They also proved a realization theorem for it, i.e., for each H\"older semicomplex $G$, there is a germ of a semialgebraic curve $(X,x_0)$ such that the Hölder Semicomplex of $(X,x_0)$ coincides with $G$. In 2000, Birbrair proposed the problem of finding and characterizing all the H\"older semicomplexes that can be realized as a Hölder Semicomplex of a real analytic curve. Solving this problem in its full generality is a notoriously hard task; however, in this paper, we provide a complete answer to this problem when we restrict it to tangential nodal curves: curves defined as the union of the zero set of a harmonic function and its tangent cone at a fixed point.
\end{abstract}

\thanks{The first named author was partially supported by CNPq-Brazil grants 303375/2025-6, 404611/2025-7 and 405618/2023-9. The second named author was financed in part by the Coordenação de Aperfeiçoamento de Pessoal de Nível Superior - Brasil (CAPES) - Finance Code 001. This work was supported by the Serrapilheira Institute (grant number Serra -- R-2110-39576).
}

\maketitle

\pagenumbering{arabic}

\section{Introduction}

In \cite{birb}, Birbrair and Fernandes provided a complete bi-Lipschitz classification of germs of semialgebraic curves (or, more generally, curves that are definable in a polynomially bounded o-minimal real structure). For that, they introduced the so called H\"older semicomplex, which is an outer bi-Lipschitz invariant associated with the germ of a semialgebraic curve $X$ at the point $x_0$, denoted by $sh(X,x_0)$. Roughly speaking, $sh(X,x_0)$ is a pair consisting of a complete graph collecting the semibranches of $X$ and a map gathering the contact orders (see Definition \ref{Def: cont order}) of each pair of semibranches. Moreover, they also defined abstractly a purely combinatorial object called H\"older semicomplex and proved a realization theorem for it, i.e., for each H\"older semicomplex there is a semialgebraic curve $X$ such that $sh(X,x_0)$ coincides with the given one.  

We learnt from Fernandes that Birbrair proposed the following problem in 2000:
\begin{problem}
Find and characterize all the H\"older semicomplexes $\Gamma$ such that there is a real analytic curve $X$ with $sh(X,x_0)=\Gamma$. 
\end{problem}

Solving Problem 1 in its full generality is a notoriously hard task. The category of general real analytic curves exhibits extreme topological and metric flexibility. 
Because of this inherent wildness, it is a necessary and fruitful strategy to restrict the realization problem to a more rigid, yet sufficiently rich, subclass of analytic curves. 

The first class of germs of analytic sets that we have in mind is based on nodal curves: curves defined as the zero set of a harmonic function in some open subset of $\R^2$. 
In this paper, we tackle Problem 1 by restricting our focus to tangential nodal sets: curves defined as the union of the zero set of a harmonic function and its tangent cone at a fixed point. Harmonic functions impose strict geometric rigidity that tames the wild behavior seen in general analytic curves while still generating non-trivial H\"older semicomplexes.  

Specifically, we determine all the semicomplexes $\Gamma$ for which there exists a nodal curve $X$ satisfying $sh(X \cup C(X, x_0), x_0) = \Gamma$. We begin by showing that the result equivalent to Theorem 4.1 in \cite{birb} is significantly simpler in this context (see Proposition \ref{Prop: simple classif 2}), as the contact orders of semibranches of nodal curves are always equal to 1. However, by considering the set $X \cup C(X, 0)$, we obtain higher contact orders. We provide a closed formula to compute the contact order between a branch and its tangent line  (see Corollary \ref{Cor: formula geral para contato}), which we define as its linear contact order. 

To bridge the gap between abstract semicomplexes and those actually realizable by tangential nodal sets, we uncover rigid arithmetic criteria that govern the distribution of linear contact orders. By analyzing the harmonic functions that define these sets, we demonstrate that the admissible linear contact orders are heavily constrained by modular arithmetic conditions on $\mathbb{Z}_n$. Specifically, for a nodal set with $n$ branches, numbered from $0$ to $n-1$, we introduce a sequence of subgroups $H_m \subset \mathbb{Z}_n$ generated by divisors of the form $n/\gcd(n, \alpha(m)-1)$, where $\alpha(m)$ is the linear contact order of the branch $m$. We prove that the distribution of linear contact orders is completely determined by a contact filtration of these subgroups.  

With these arithmetic criteria established, we formally define a filtration semicomplex (see Definition \ref{Def: filtration semicomplex}), which is an algebraic object admitting a realization theorem (see Theorem \ref{Thm: realization theorem}). The proof of our realization theorem (Theorem \ref{Thm: realization theorem}) then proceeds constructively. Given an abstract filtration semicomplex that strictly satisfies these algebraic and divisibility conditions, we explicitly construct a complex polynomial $F$ such that the tangential nodal set defined by its real part exactly realizes the prescribed filtration. Depending on whether the maximum linear contact order is finite or infinite, the terms of $F$ are carefully chosen so that their degrees and coefficients dictate the contact orders via our closed formula, effectively avoiding specific degenerate trigonometric values.

The paper is organized as follows. Section \ref{Section 2 - preliminaries} provides preliminary definitions regarding the outer Lipschitz geometry of nodal sets. Section \ref{Section 3 - lin cont orders} establishes the closed formula for computing linear contact orders. Finally, Section \ref{Sec 4 - applications} defines the associated semicomplex and semicomplex filtration, proves the realization theorem, and establishes a version of Theorem \ref{Thm: generalização do principal} for unions of two nodal sets.

\textbf{Acknowledgements.} The authors acknowledge the use of Claude AI for assistance with language editing and the improvement of the exposition of this paper. All mathematical arguments and results were reviewed and verified by the authors, who takes full responsibility for the contents of the article.

\section{Preliminaries}\label{Section 2 - preliminaries}

\begin{definition}
    Let $f\colon U \subset \R^2 \to \R$ be a function. We define the \textbf{zero level of $f$} as the set $Z(f)=\{p \in U \mid f(p) = 0\}$. Moreover, we call $f$ \textbf{harmonic} if $\frac{\partial^2 f}{\partial x^2} + \frac{\partial^2 f}{\partial y^2} = 0$.\\
\end{definition}


\begin{remark}\label{Rem: harmonic is locally real analytic}
    A relevant known fact in complex analysis is that every harmonic function can be locally written as the real part of a holomorphic function (see Theorem 2.30 in \cite{conway}), and if the harmonic function is defined in a ball (generally, a simply connected domain), the harmonic function is actually the real part of a holomorphic function. In particular, the zero level of a harmonic function is a real analytic set.
\end{remark}

\begin{definition}\label{Def: singularity}
    Let $f\colon U \subset \R^2 \to \R$ be a harmonic function. We call a point $p \in U$ a \textbf{singularity} of $f$ if $\frac{\partial f}{\partial x}(p) = \frac{\partial f}{\partial y}(p) = 0$. \\
\end{definition}

\begin{definition}\label{Def: nodal set}
    A subset $X \subseteq \R^2$ is called a \textbf{nodal set} if $X = Z(f)$, where $f\colon U \subseteq \R^2 \to \R$ is a harmonic function. \\
\end{definition}

Since it is well-known that any harmonic function is an analytic function, we have the following:
\begin{proposition}\label{Prop: analytic curve}
    If $f\colon U \subset \R^2 \to \R$ is a harmonic function, then $Z(f)$ is locally an analytic curve around every point $p\in Z(f)$.
\end{proposition}

    \begin{remark}\label{Rem: subanalytics are pol bound}
        It is known that the set of globally subanalytic subsets of $\R^n$ is a polynomially bounded o-minimal structure (see \cite{dries} for a proof). Hence, the main theorems regarding the structure of definable o-minimal sets apply to nodal sets of harmonic functions, since these nodal sets are real analytic sets. Then the Local Conic Structure (see Theorem 4.10 in \cite{OMin}) holds for nodal sets. 
    \end{remark}

\begin{theorem}[Local Conic Structure]\label{Thm: local conic structure}
    Let $X \subseteq \R^n$ be a closed definable set in an o-minimal structure and $a \in X$. Then there is an $\varepsilon > 0$ and a definable homeomorphism $h\colon Cone(a, S_{\varepsilon}(a) \cap X )  \to \overline{B}_{\varepsilon}(a) \cap X$ such that $h(x) = x$ for $x \in S_{\varepsilon}(a) \cap X$ and $|h(x) - a| = |x - a|$. 
\end{theorem}

\begin{theorem}[One Dimensional Pancake Decomposition Theorem]\label{Thm: pancake decomposition}
    Let $X \subseteq \R^n$ be a definable closed curve. Given $a \in X$, there are $\Gamma_1,...,\Gamma_m$ closed definable sets such that:
    There is $\varepsilon >0$ such that \\
    $\overline{B}_\varepsilon (a) \cap X = \bigcup\limits_{k=1}^m \Gamma_k$;\\
    $\Gamma_i \cap \Gamma_j = \{a\}$ for $i \not = j$;\\
    $\Gamma_i \cap S_\delta (a)$ is a singleton for every $\delta \in [0, \varepsilon)$. 
\end{theorem}

The proof of this theorem is a direct consequence of the Local Conic Structure for curves. 

\begin{remark}
    The Pancake Decomposition Theorem is even more general, it states that each set $\Gamma_k$ is locally normally embedded, and it does not require $X$ to be a curve. However, this generalization will not be used in the present paper. 
\end{remark}

\begin{definition}
    We call each $\Gamma_k$ in Theorem \ref{Thm: pancake decomposition} a \textbf{semibranch}. Note that each semibranch is definably homeomorphic to $[0,\varepsilon)$. 
\end{definition}

\begin{lemma}[Analytic Curve Selection Lemma]\label{Lem: curve selection} 
    If $X$ is semianalytic in $\R^n$ and $0 \in \overline{X}$, then there is an analytic arc $\gamma\colon [0, \varepsilon) \to X$ such that $\gamma(0) = 0$ and $\gamma((0,\varepsilon)) \subset X$. 
\end{lemma}

For a proof of the analytic curve selection lemma, see Lemma 2.2.3 in \cite{guillerm}. 

\begin{definition}\label{Def: cont order}
    Let $(X,0)$ be the germ of a nodal set. Let $\Gamma_1$ and $\Gamma_2$ be two semibranches of $(X,0)$. The \textbf{contact order} of $\Gamma_1$ and $\Gamma_2$, denoted by $\Cont(\Gamma_1,\Gamma_2)$, is defined as
    $$\lim\limits_{r \to 0^+} \frac{\log d(\Gamma_1\cap S_r(0),\Gamma_2\cap S_r(0))}{\log r},$$
    which is well-defined by Theorem \ref{Thm: pancake decomposition}.
\end{definition}

\begin{remark}\label{Rem: cont order vs least exponent}
    We set, by convention, $\Cont(\Gamma_1,\Gamma_2)=\infty$ when $\Gamma_1 = \Gamma_2$. Notice that $\Cont(\Gamma_1,\Gamma_2)$, as in Definition \ref{Def: cont order}, coincides with the least exponent of the Newton-Puiseux series of $h(r) = d(\Gamma_1\cap S_r(0),\Gamma_2\cap S_r(0))$ around $0$, which exists for globally subanalytic functions (recall Remark \ref{Rem: subanalytics are pol bound}).\\
\end{remark}

\begin{remark}\label{Rem: isosceles property}
 Another important property is the isosceles property of the contact order. It states that if you have three semibranches $\Gamma_1, \Gamma_2, \Gamma_3$ and their respective contact orders, the two smallest contact orders are equal. In particular, $$Cont(\Gamma_1,\Gamma_2) \ge \min\{Cont(\Gamma_1,\Gamma_3),Cont(\Gamma_3,\Gamma_2)\}.$$ This property was first proved for the semialgebraic category in \cite{birb}.
 \end{remark}

The following Proposition provides two ways of parameterizing two semibranches without changing their contact order, i.e., they are both parametrized either by the distance to the origin or by one of the coordinates.

\begin{proposition}\label{Prop: parametrization by the first coordinate}
    
    Let 
    $\alpha_1, \beta_1$ and $\alpha_2, \beta_2$ be parametrizations of semibranches with no vertical tangent cones of distinct analytic (or definable in a real polynomially bounded o-minimal structure) curves such that $$Im(\alpha_i)=Im(\beta_i), $$for $i=1,2$ and $$||\alpha_1(t)|| = t = ||\alpha_2(t)||$$ for any small enough $t$. If $\beta_1$ and $\beta_2$ are parametrized by the first coordinate, then $||\alpha_1(t) - \alpha_2(t)||$ and $||\beta_1(t) - \beta_2(t)||$ have the same least non-vanishing exponent of the Newton-Puiseux series at the origin.
\end{proposition}

\begin{proof}

    We can write $$||\alpha_1(t) - \alpha_2(t)|| = at^\alpha + o(t^\alpha)$$ and $$||\beta_1(t) - \beta_2(t)|| = bt^\beta + o(t^\beta),$$ for some constants $a,b\in \R^*$. We shall prove that $\alpha = \beta$. \\

    Suppose, initially, that the parametrized semibranches have equal tangent cones, which are not vertical (first coordinate $0$), because there is an analytic parametrization by the first coordinate by hypothesis. 
    
    For a sufficiently small $t$, consider $p_1 = \alpha_1(t)=(u,v)$, $p_2 = \beta_2(u)$ and $p_3 = \alpha_2(t)$. Since we can assume the parametrizations are bijective, it follows that $\beta_1(u) = p_1$.
    
    Consider the triangle with vertices $p_1, p_2$ and $p_3$ and angles $\theta_1$, $\theta_2$ and $\theta_3$ in those respective vertices. Suppose that $\frac{||p_1 - p_3||}{||p_1 - p_2||} \to 0$, then $\frac{\sin(\theta_2)}{\sin(\theta_3)} \to 0$ by the law of sines, hence $\theta_2 \to 0$. However, since $p_2 = \beta_2(u) = (u, \Tilde{\beta}(u))$, the segment that connects $p_1$ and $p_2$ is vertical, and if $\theta_2 \to 0$, the segment $L$ that connects $p_2$ and $p_3$ must tend to a vertical one. But since $L$ is a secant of $Im(\alpha_2)$, as $t \to 0$, the direction of this secant must tend to the direction of the derivative in zero, which is not vertical. Therefore, $\theta_2 \not \to 0$.

    Analogously, suppose that $\frac{||p_1 - p_2||}{||p_1 - p_3||} \to 0$, then $\frac{\sin(\theta_3)}{\sin(\theta_2)}\to 0$, but this implies $\theta_3 \to 0$. Therefore, if $O$ is the origin of $\R^2$, then $p_2\widehat{p_3}O\to 0$ because the secant ray from $p_2$ through $p_3$ approaches the tangent cone of $\alpha_2$. Thus, as 
    $$p_1\widehat{p_3}O = \theta_3 \pm p_2\widehat{p_3}O,$$ 
it follows that $p_1\widehat{p_3}O\to 0$, which is a contradiction since $p_1$ and $p_3$ belong to the same sphere centered at $O$. 

    Therefore, the functions $h(t) = ||p_1 - p_2||$ and $g(t) = ||p_1 - p_3||$ have the same least non-vanishing exponent of the Newton-Puiseux series at the origin. Notice that $h(t) = ||\beta_1(u(t)) - \beta_2(u(t))||$ and $g(t)=at^\alpha + o(t^\alpha)$. 
    
    It only remains to prove that the exponent of $h(t)$ is $\beta$. However, $h(t) = r(u(t))$ where $r(t)= ||\beta_X(t) - \beta_Y(t)|| = bt^\beta + o(t^\beta)$. But as $\alpha_1'(0)$ is not vertical, then $u'(0) \not = 0$ and hence $u(t) = at + o(t)$. Hence $$r(u(t)) = r(at + o(t)) = b(at + o(t))^{\beta} + o(t^{\beta}) = abt^\beta + o(t^\beta).$$ This means that the least non-vanishing exponent of $r(u(t))$ is exactly the least non-vanishing exponent of $r(t)$.

\end{proof}

\begin{definition}\label{Def: angle}
    Let $\Gamma_1,\Gamma_2\colon [0,\epsilon) \to \R^2$ be germs of regular analytic maps at $0$ such that $\Gamma_1(0) = \Gamma_2(0) = p$. We define the \textbf{angle between $\Gamma_1$ and $\Gamma_2$ at $p$}, denoted by $\measuredangle(\Gamma_1,\Gamma_2)$, to be the angle between $\Gamma_1'(0)$ and $\Gamma_2'(0)$. It can be proved that $Cont(\Gamma_1, \Gamma_2) > 1$ if and only if $\measuredangle(\Gamma_1,\Gamma_2) = 0$.\\
\end{definition}

\begin{remark}\label{Rem: contato}
One can prove that the contact order between the semibranches $\{(x,y)\in \R^2\mid y = a_nx^n + a_{n+1}x^{n+1}+..., x\ge 0\}$ and $\{(x,0)\in \R^2\mid x\ge 0\}$ is exactly $n$. 
\end{remark}

The following lemma is a consequence of the invariance of contact orders under bi-Lipschitz maps, which is proved in \cite{Edson}. In order to be self-contained, we also present a proof here. 
\begin{lemma}\label{Lem: holomorohic preserves} 
    If $F\colon B(0, \varepsilon) \to \C$ is a holomorphic function such that $F'(0) \not = 0$ and $F(0)=0$, then for any two curves passing through the origin, the map $F$ preserves the contact order. 
\end{lemma}

\begin{proof}
    If $F'(0) \not = 0$, then there exists $\delta > 0$ such that $F:B(0, \delta) \to F(B(0, \delta))\subset \C$ is a bi-Lipschitz mapping, and bi-Lipschitz mappings preserves contact order. Indeed, suppose that $\gamma_1$ and $\gamma_2$ are two differentiable curves from $[0,\varepsilon)$ to $\C$ such that $\gamma_1(0)=\gamma_2(0)=0$. By Remark \ref{Rem: harmonic is locally real analytic} and Theorem \ref{Thm: local conic structure}, we may assume that $||\gamma_i(t)||=t$ for $t \in [0,\varepsilon)$ and $i=1,2$. Then $$\frac{1}{C} \le \frac{||F(\gamma_1(t)) - F(\gamma_2(t))||}{||\gamma_1(t) - \gamma_2(t)||}\le C.$$ 
So, the least exponent of $h(t) = ||F(\gamma_1(t)) - F(\gamma_2(t))||$ is the contact order of the gamma curves, because we consider the triangle $F(\Delta)$ composed of $F(\gamma_1(t))$, $F(\gamma_2(t))$, and $F(x)$, which is a point in $F(\gamma_2([0,\varepsilon))$ such that $||F(x)|| = ||F(\gamma_1(t))||$. Because $F$ is conformal in an open ball centered at the origin, the angles of $F(\Delta)$ are exactly the angles of the triangle $\Delta$ composed of $\gamma_1(t)$, $\gamma_2(t)$, and $x$. The result follows from a proof similar to Proposition \ref{Prop: parametrization by the first coordinate}. Then, it follows from Remark \ref{Rem: cont order vs least exponent} that holomorphic functions with non-vanishing derivative preserve contact order.
\end{proof}

\begin{definition}\label{Def: adjacent semibranches}
    Let $f\colon U \subset \R^2 \to \R$ be a harmonic function on an open set $U$, and let $\Gamma_1, \Gamma_2$ be two distinct semibranches of $Z(f)$. We define $\Gamma_1$ and $\Gamma_2$ as \textbf{opposite semibranches} if $C(\Gamma_1,0)\cup C(\Gamma_2,0)$ is a line.
    Moreover, we call $\Gamma_1$ and $\Gamma_2$ \textbf{adjacent} if they are not opposite and $$C(\Gamma,0) \not\subset conv.hull(C(\Gamma_1,0)\cup C(\Gamma_2,0))$$ for any other semibranch $\Gamma$ of $Z(f)$. The union of opposite semibranches of a nodal set is called a \textbf{branch}.  
\end{definition}

\begin{lemma}\label{Lemma: Lema da Topologia de conjuntos nodais}
     Let $f\colon U \subset \R^2 \to \R$ be as in Definition \ref{Def: adjacent semibranches}. Then, for every point $p \in Z(f)$, there are a neighbourhood $B$ of $p$ and a positive integer $n$ such that $Z(f) \cap B$ contains exactly $2n$ semibranches intercepting only at $p$. Moreover, each branch is an analytic regular curve, and each pair $\Gamma_1, \Gamma_2$ of adjacent semibranches satisfies that $\measuredangle(\Gamma_1, \Gamma_2) = \frac{\pi}{n}$. There is also a complex analytic homeomorphism such that $h\colon B \to B(0,1)$ and $h(\gamma_k(t)) = t e^{i\frac{(k-1)\pi}{n}}$. 
\end{lemma}

For a proof of this theorem, see Theorem 1 in \cite{estrutura_harmonica}. 

\begin{remark}\label{Rem: number of branches equals order}
 A consequence of the proof of Lemma \ref{Lemma: Lema da Topologia de conjuntos nodais} in \cite{estrutura_harmonica} is that if $f(x, y) = \Re(F(z))$ with $F$ being holomorphic, then the number of branches is exactly the order of $F$, and if the least coefficient that does not vanish is $i$, then there is a branch tangent to the line $y = 0$.
 \end{remark}

\begin{remark}\label{Rem: cont ord of seminrahces of nodal set}
    The contact order of two semibranches of a nodal set is always equal to $1$.
\end{remark}

\begin{theorem}\label{Thm: simple classif}
Two nodal sets $X$ and $Y$ of $\R^2$ are outer bi-Lipschitz homeomorphic as germs at a point if and only if they have the same number of branches. More generally, if $X, Y_1,...,Y_n$ are nodal sets and $Y = \bigcup_{k = 1}^n Y_k$, then the germ of $X$ is outer bi-Lipschitz homeomorphic to the germ of $Y$ if and only if they have the same number of branches and any two different semibranches of $Y$ have a positive angle (see Remark \ref{Rem: cont ord of seminrahces of nodal set}). 
\end{theorem}

\begin{proof}
    Notice that it is enough to prove the more general version of the Theorem. By Theorem 4.1 in \cite{birb}, it is enough to prove that the H\"older Semicomplexes of $X$ and $Y$ are isomorphic. By Lemma \ref{Lemma: Lema da Topologia de conjuntos nodais}, any two distinct semibranches $\Gamma_1$ and $\Gamma_2$ of $X$ have $\measuredangle(\Gamma_1,\Gamma_2) > 0$. By Definition \ref{Def: cont order} this implies that $Cont(\Gamma_1, \Gamma_2) = 1$. Then, for the H\"older Semicomplex of $X$ to be isomorphic to the H\"older Semicomplex of $Y$, we need $X$ to have the same number of semibranches as $Y$, and for any two distinct semibranches $\Gamma_1', \Gamma_2'$  of $Y$, we need to have $Cont(\Gamma_1',\Gamma_2') = 1$, but this is equivalent to $\measuredangle(\Gamma_1',\Gamma_2') >0$ (see Definition \ref{Def: cont order}).
\end{proof}

\section{Linear Contact Orders of Nodal Sets}\label{Section 3 - lin cont orders}

Now we are going to find the contact order of a semibranch of a nodal set with its tangent cone at the origin (see Definition \ref{Def: linear contact order}). Along this section, we will assume that $f\colon B(0,r)\to \R$ is a harmonic function. By Remark \ref{Rem: harmonic is locally real analytic}, we have $f(x, y) = \Re(F(x + iy))$ for some holomorphic function $F:B(0, r) \to \C$, where $F(z) = c_nz^n + c_{n+1}z^{n+1}+...$ with $c_n \not = 0$. We may assume, without loss of generality, that $$F(x, y)=iz^n +c^*_{n+1}z^{n+1}+\ldots \textrm{.}$$ Indeed, otherwise, we just apply a rotation followed by a multiplication by a real positive constant (therefore, a bi-Lipschitz map, which preserves the contact orders) on $\C$ as a real vector space to make $c_n = i$ and one of the semibranches of $X$ tangent to $\{(x, 0) \in \R^2 : x > 0\}$ (see Remark \ref{Rem: number of branches equals order}).

\begin{lemma}\label{Lema: expressao das funcoes harmonicas}
    One can write $f$ as $$f = -u_{2n} + \sum_{m > n} a_mu_{1m} + b_mu_{2m}$$ for $a_m, b_m \in \R$ and $$u_{1m}(x, y) = \sum_{k = 0}^{\lfloor\frac{m}{2}\rfloor} \binom{m}{2k}(-1)^kx^{m-2k}y^{2k}$$ and $$u_{2m}(x,y)= \sum_{k = 0}^{\lfloor\frac{m - 1}{2}\rfloor} \binom{m}{2k + 1}(-1)^k x^{m - 2k - 1} y^{2k + 1},$$ 
    where the $u_{1m}$ and $u_{2m}$ are the first and second $m$-th elementary harmonic polynomials, respectively. \\
\end{lemma}

\begin{proof}
    Since we assume $f(x, y) = \Re(iz^n + c^*_{n+1}z^{n+1}+...)$, $f(x, y) = \Re(iz^n) +\sum\limits_{k > n} \Re(c^*_kz^k)$, but if $c^*_k = a_k - ib_k$ for $a_k, b_k \in \R$, then $f(x, y) = - \Im(z^n) + \sum\limits_{k > n} a_k\Re(z^k) + b_k\Im(z^k)$, but $\Re(z^k)$ and $\Im(z^k)$ are exactly $u_{1k}$ and $u_{2k}$ respectively, by the binomial theorem. \\
\end{proof}

\begin{definition}\label{Def: linear contact order}
    Let $\Gamma$ be a semibranch of $Z(f)$ with tangent cone at the origin $C(\Gamma,0)$. We define the \textbf{linear contact order }of $\Gamma$ to be $LC(\Gamma,0) = Cont(\Gamma,C(\Gamma,0))$. Notice that the linear contact order of a semibranch could be infinite.
\end{definition}

We can now present a formula to compute the linear contact order of a semibranch in the case where its tangent cone is the non-negative part of the real axis.

\begin{theorem}\label{Teo: formula da ordem de contato} 
 Let $f(x, y) = - u_{2n} + \sum\limits_{k > n} a_ku_{1k} + b_ku_{2k}$ be as in Lemma \ref{Lema: expressao das funcoes harmonicas} and $\Gamma$ a semibranch of $Z(f)$ with tangent cone at the origin $C(\Gamma,0) = \{(x,0)\in \R^2\mid x\ge 0\}$. Then $$LC(\Gamma,0) = \inf\{k \in \N : a_k \not = 0\} - n + 1.$$ 
 We convention $\inf \emptyset = \infty$
\end{theorem}

\begin{proof}
    Lemma \ref{Lem: curve selection} implies that the branch $\Gamma$ can be locally expressed as the graph of an analytic function, i.e., $\Gamma$ equals $y = c_1x + c_2x^2 +...$ at the origin. Since $C(\Gamma,0) = \{(x,0)\in \R^2\mid x\ge 0\}$, it follows that $c_1 = 0$. Set $m = \inf\{k \in \N : a_k \not = 0\}$. We will prove that $c_2 =...=c_{m-n} = 0$ and $c_{m - n + 1} \not = 0$.

    If $f(x, y) = 0$, then $u_{2n}(x,y)= a_mu_{1m} + \sum\limits_{r > m} a_ru_{1r} + \sum\limits_{r > n} b_ru_{2r}$. This means that
    \begin{equation}\label{Eq: cont ord}
        \begin{split}
            \sum_{k = 0}^{\lfloor\frac{n - 1}{2}\rfloor} \binom{n}{2k + 1}(-1)^k x^{n - 2k - 1} y^{2k + 1}  = a_m\sum_{k = 0}^{\lfloor\frac{m}{2}\rfloor} \binom{m}{2k}(-1)^kx^{m-2k}y^{2k} +  \\\\ \sum_{r > m} a_r\sum_{k = 0}^{\lfloor\frac{r}{2}\rfloor} \binom{r}{2k}(-1)^kx^{r-2k}y^{2k} + \sum_{r > n} b_r \sum_{k = 0}^{\lfloor\frac{r - 1}{2}\rfloor} \binom{r}{2k + 1}(-1)^k x^{r - 2k - 1} y^{2k + 1}.
        \end{split}
    \end{equation}

    By induction, consider $p$ with $2 \le p < m - n + 1$ and $y = c_px^p +...$, the least exponent term of 
    $$\binom{n}{2k + 1}(-1)^k x^{n - 2k - 1} (c_px^p + ...)^{2k + 1}$$ is $$\binom{n}{2k +1}(-1)^k c_p^{2k+1}x^{n + (p - 1)(2k + 1)}.$$
    Thus, the least exponent term of $$\sum\limits_{k = 0}^{\lfloor\frac{n - 1}{2}\rfloor} \binom{n}{2k + 1}(-1)^k x^{n - 2k - 1} (c_px^p+...)^{2k + 1}$$ is when $k = 0$, hence it is $nb_{n+ 1}c_px^{n + p - 1}$. 

    The least exponent term of $$\binom{m}{2k}(-1)^kx^{m-2k}(c_px^p+...)^{2k}$$ is $$\binom{m}{2k}(-1)^kc_p^{2k}x^{m+2k(p-1)}.$$ This means that the least exponent of $$a_m\sum\limits_{k = 0}^{\lfloor\frac{m}{2}\rfloor} \binom{m}{2k}(-1)^kx^{m-2k}y^{2k}$$ is when $k = 0$, hence it is $a_mx^m$. Applying the same argument, we find the least possible exponent term of $$\sum\limits_{r > m} a_r\sum\limits_{k = 0}^{\lfloor\frac{r}{2}\rfloor} \binom{r}{2k}(-1)^kx^{r-2k}y^{2k}$$ that is $a_{m+1}x^{m+1}$ and the least possible exponent term of $$\sum\limits_{r > n} b_r \sum\limits_{k = 0}^{\lfloor\frac{r - 1}{2}\rfloor} \binom{r}{2k + 1}(-1)^k x^{r - 2k - 1} y^{2k + 1}$$ that is $(n+1)c_px^{n+p}$. Finally, if $p < m - n + 1$, then $n + p - 1<m$, then $n + p - 1$ is smaller than every degree on the right hand side of Equation \ref{Eq: cont ord}, which implies that $c_p = 0$. \\

    It only remains to prove that $c_{m - n + 1} \not = 0$. As proved previously, $y = c_{m - n + 1}x^{m - n + 1}+...$. Thus the least exponent on the left hand side of Equation \ref{Eq: cont ord} is $nc_{m - n+1}x^m$. Analogously, the least exponent on the right hand side of Equation \ref{Eq: cont ord} is either $a_mx^m$ or $(n+1)c_{m-n+1}x^{n + (m-n+1)} = (n+1)c_{m-n+1}x^{m+1}$, hence it is $a_mx^m$. Since $a_m \not = 0$, one obtains $nc_{m-n+1} = a_m$, which implies $c_{m-n+1} = \frac{a_m}{n}$. In particular, $c_{m - n + 1} \not = 0$.\\
    
     In other words, the semibranch in question is locally $y = c_{m-n+1}x^{m-n+1} +... $ with $c_{m-n+1}\not = 0$, and then the contact order with its tangent line is exactly $m - n + 1$ (see Remark \ref{Rem: contato}). \\
\end{proof}

\begin{corollary}\label{Cor: formula geral para contato}
    Let $f = - u_{2n} + \sum\limits_{k > n} a_ku_{1k} - b_ku_{2k}$ be as in Lemma \ref{Lema: expressao das funcoes harmonicas} and $\Gamma$ a semibranch of $Z(f)$, where $C(\Gamma,0)$ is the ray from the origin with direction $e^{\frac{r\pi i}{n}}$, for some $r \in \{0,1,\ldots, 2n-1\}$. Then $$LC(\Gamma,0) = \inf\left\{ k \in \N : a_k\cos\left(\frac{r\pi k}{n}\right) \not =  b_k \sin\left(\frac{r\pi k}{n}\right) \right\} - n + 1.$$ 
\end{corollary}

\begin{proof}
    Let $F(z) = iz^n + \sum\limits_{k > n} c_kz^k$. It follows that $a_k = \Re(c_k)$ and $b_k = \Im(c_k)$. Applying the change of coordinates $z\mapsto e^{\frac{r \pi i}{n}} z$ (which preserves the contact orders), we obtain a rotation of the complex plane such that $$F(e^{\frac{r \pi i}{n}} z) = \pm i z^n + \sum\limits_{k>n} c_ke^{\frac{rk\pi i}{n}} z^k.$$

    The nodal set associated with the real part of this harmonic function is the former nodal set rotated in $\frac{r\pi}{n}$ radians. Therefore, Theorem \ref{Teo: formula da ordem de contato} implies that 
    $$LC(\Gamma,0) = \inf\left\{k \in \N :\Re(c_ke^{\frac{rk\pi i}{n}}) \not =  0\right\} - n + 1,$$
    but $\Re(c_ke^{\frac{rk\pi i}{n}}) = a_k\cos\left(\frac{rk\pi}{n}\right) - b_k\sin\left(\frac{rk\pi}{n}\right)$. 
\end{proof}

\begin{corollary}\label{Cor: opposite}
    If $\Gamma_1$ and $\Gamma_2$ are opposite semibranches of $Z(f)$, then $$LC(\Gamma_1,0) =LC(\Gamma_2,0).$$ In words, opposite semibranches have the same linear contact order (see Definition \ref{Def: linear contact order}).
\end{corollary}
\begin{proof}
    If $\Gamma_1$ and $\Gamma_2$ are opposite semibranches of $Z(f)$, then by Lemma \ref{Lemma: Lema da Topologia de conjuntos nodais} there is a positive integer $n$ such that, up to renaming the semibranches, we may assume that they have directions $e^\frac{r\pi i}{n}$ and $e^\frac{(r+n)\pi i}{n}$, respectively, for some $r\in \{0,1,\ldots 2n -1\}$. Finally, using trigonometric identities and Corollary \ref{Cor: formula geral para contato}, we obtain the desired result.
\end{proof}

\begin{remark}
    By Lemma \ref{Lemma: Lema da Topologia de conjuntos nodais}, to compute the linear contact orders of the semibranches of a nodal set $X$, we would have to compute $2n$ contact orders, for some positive integer $n$. However, Corollary \ref{Cor: opposite} implies that only half of those contact orders must be considered. Indeed, it shows that in the harmonic case, contact orders behave similar to angles. The importance of linear contact orders of $X$ will become clearer later, especially in Section \ref{Sec 4 - applications}. 
\end{remark}

\begin{remark}
    One can more easily prove Corollary \ref{Cor: opposite} using only the analyticity of the branch guaranteed by Lemma \ref{Lemma: Lema da Topologia de conjuntos nodais}.
\end{remark}

Corollary \ref{Cor: opposite} allows us to formalize the following notion. 
\begin{definition}\label{Def: branch cont order}
    The \textbf{linear contact order of a branch} is the linear contact order of any of its two (opposite) semibranches. 
\end{definition}

The next examples show the diversity we could have regarding linear contact orders of nodal sets.

\begin{example}
    The nodal set generated by $f(x, y) = -xy + x^3 - 3xy^2$ has two branches at the origin, one with infinite linear contact order (its tangent line is the vertical one) and another with finite linear contact order (its tangent line is the horizontal axis), as in Figure \ref{Fig: two branches}. 
    
    \begin{figure}[h!]
        \centering
        \includegraphics[scale=0.12]{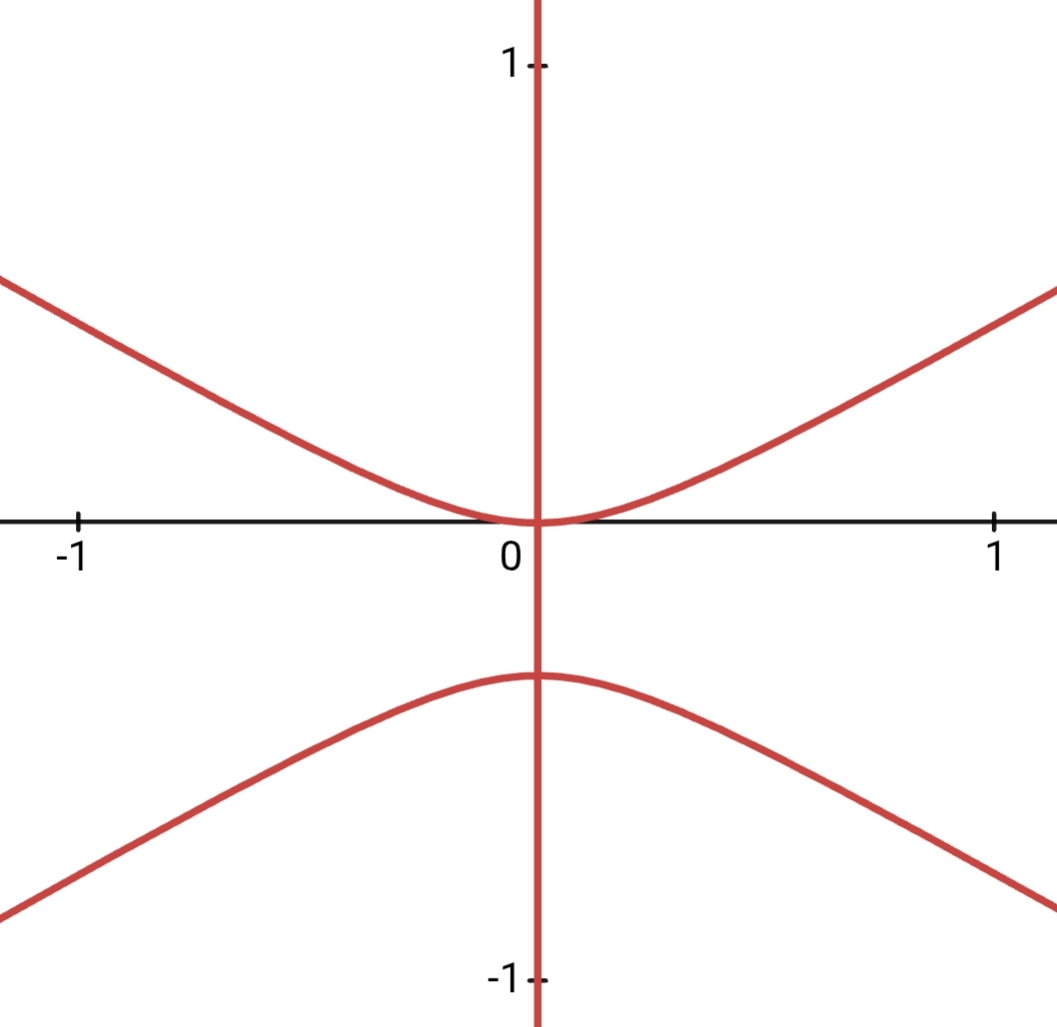}
        \caption{Nodal set germ with two branches.}
        \label{Fig: two branches}
    \end{figure}
\end{example}

As the next two examples show the set of finite linear contact orders of a nodal set is not a singleton in general.

\begin{example}
    The nodal set generated by $f(x, y) = \Re(iz^3 + (\sqrt{3} + i)z^4 + (1 + i)z^5)$ has 3 branches, one with linear contact order $2$ (its tangent line is the $x$-axis) and another with linear contact order $3$ (its tangent line in the direction $e^\frac{\pi i}{3}$), as in Figure \ref{Fig: three branches}.
    \begin{figure}[h!]
        \centering
        \includegraphics[scale=0.12]{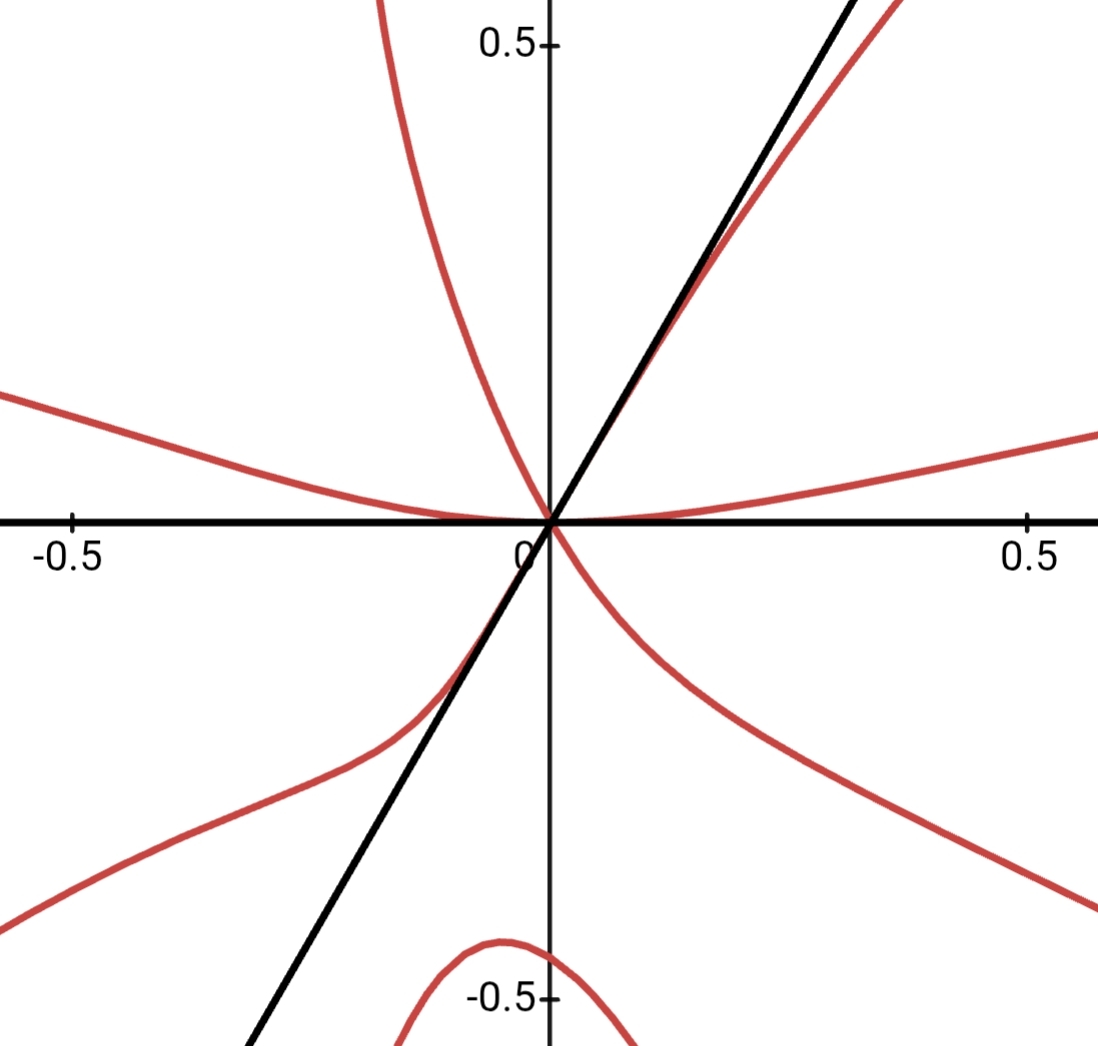}
        \caption{Nodal set germ with three branches.}
        \label{Fig: three branches}
    \end{figure}
\end{example}

Even knowing the linear contact orders of two nodal sets through Corollary \ref{Cor: formula geral para contato}, it is hard to compute the contact order between two semibranches of different nodal sets, as seen in Example \ref{Exam: union of two nodal sets} below. Unfortunately, contrary to what one might think at first, the Isosceles Property (see Remark \ref{Rem: isosceles property}) cannot be used to help us with this computation.

\begin{example}\label{Exam: union of two nodal sets}
    Consider two harmonic functions: $f_1(x,y)= -y + xy + x^2 - y^2$ (in red) and $f_2(x,y)= - y + x^2 - y^2$ (in blue), then the linear contact order of each branch is $2$, but the contact order between the two branches is not $2$. Indeed, if $$y_1 = a_2x^2 + a_3x^3 +...,$$ then $$a_2x^2 + a_3x^3+... = x(a_2x^2 + a_3x^3+...) + x^2 - (a_2x^2 + a_3x^3+...)^2,$$ and this implies $a_2 = 1$ and $a_3 = 1$. Analogously, if $$y_2 = b_2x^2 + b_3x^3 +.
..,$$ then $$b_2x^2 + b_3x^3 +... = x^2 - (b_2x^2 + b_3x^3 +...)^2,$$ which implies $b_2 = 1$ and $b_3 = 0$. This means that $y_1 - y_2$ has order $3$ by Proposition \ref{Prop: parametrization by the first coordinate}.\\

\begin{center}
    \includegraphics[scale=0.12]{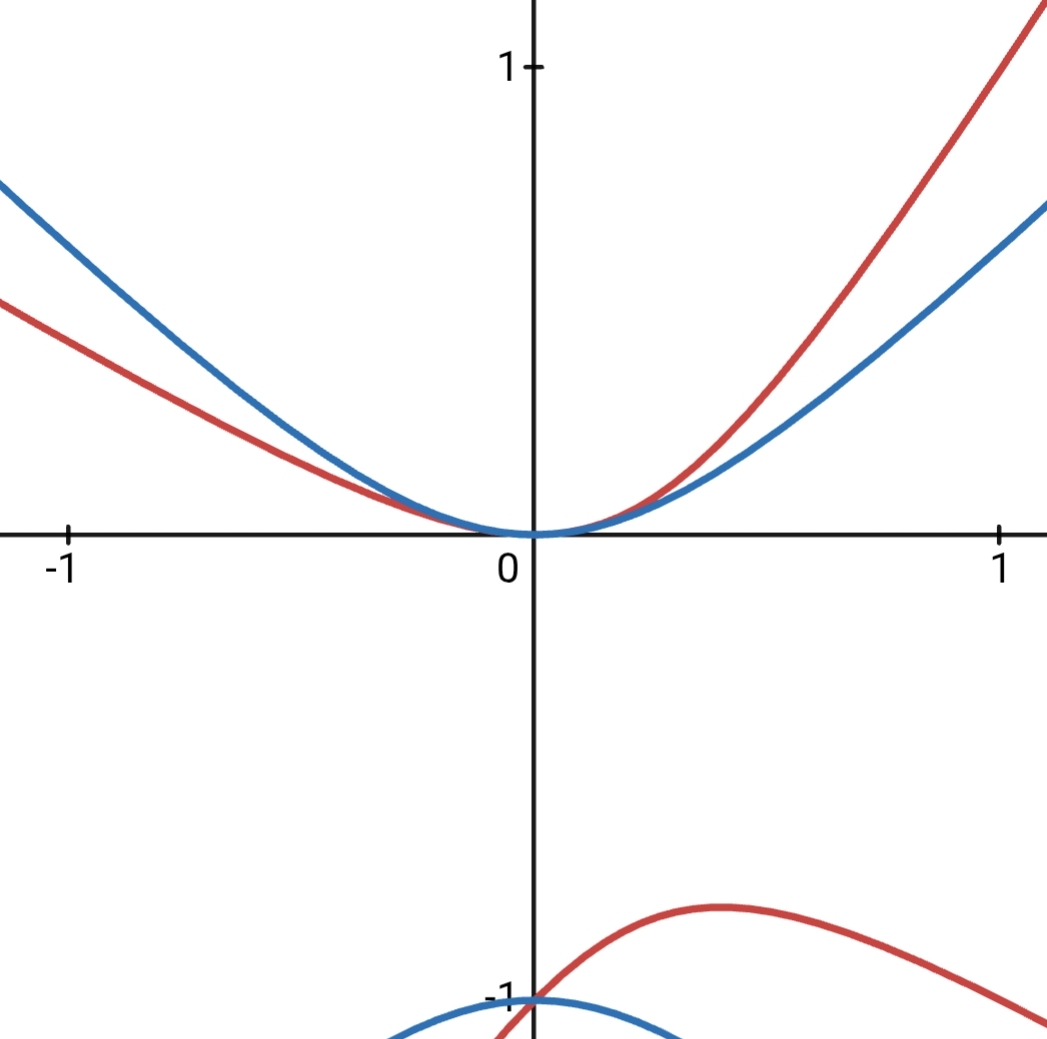}
\end{center}
\end{example}

\section{Applications of the formula for the linear contact order}\label{Sec 4 - applications}

By Theorem \ref{Teo: formula da ordem de contato}, the linear contact order of a nodal set is either a positive integer greater than 1 or infinite. We will use the following notation: $\N_{\infty} := \N \cup \{\infty\}$. 

\begin{definition}
    Denote by $\Z_n$ the set $\{0,1,...,n-1\}$ with the group structure induced by the sum modulo $n$. Let $\alpha\colon \Z_n\to \N_{\infty}$ be a function, then we call the pair $(\Z_n, \alpha)$ a \textbf{semicomplex}, if $\text{Im}(\alpha) \subseteq \{2,3,...\}\cup \{\infty\}$. Since the domain of $\alpha$ is $\Z_n$, without loss of generality, we can call $\alpha$ the semicomplex. 
\end{definition}

Let $X$ be a nodal set with $n$ branches. We assume along this section that the semibranches of $X$ are $\Gamma_0,\Gamma_1,...,\Gamma_{n-1}$, and we may assume that $\Gamma_k$ has a tangent cone in the direction of $e^{\frac{k \pi i}{n}}$. Nevertheless, up to apply a rotation, we may assume without loss of generality, that the maximum linear contact order is realized in $\Gamma_0$.

\begin{definition}\label{Def: assoc semicomplex}
    Let $X$ be a nodal set. Then the set $as(X) := (\Z_n, \alpha)$ is called the \textbf{associated semicomplex} if $\alpha\colon \Z_n \to \N_{\infty}$ is a function such that $$\alpha(k) = LC(\Gamma_k, 0).$$

    If a semicomplex is associated to a nodal set, then it is a \textbf{nodal semicomplex}
    \end{definition}

\begin{definition}
    An \textbf{isomorphism of semicomplexes} $(\Z_n, \alpha_X)$ and $(\Z_m, \alpha_Y)$ is a bijection $\sigma\colon \Z_n \to \Z_m$ such that $\alpha_Y(\sigma(k)) = \alpha_X(k)$. 
\end{definition}

By the main theorem of \cite{birb}, the associated semicomplex of a nodal set $X$ is a complete invariant for the bi-Lipschitz equivalence for $X\cup C(X,0)$, as the next proposition shows. 

\begin{proposition}\label{Prop: simple classif 2}
    Let $X$ and $Y$ be nodal sets. Their associated semicomplexes are isomorphic if and only if $X \cup C(X, 0)$ and $Y \cup C(Y, 0)$ are bi-Lipschitz equivalent.
\end{proposition}
\begin{proof}

    By the main theorem of \cite{birb}, we need only to prove that the H\"older Semicomplexes of $X' = X \cup C(X, 0)$ and $Y' = Y \cup C(Y, 0)$ are isomorphic if and only if $as(X)$ and $as(Y)$ are isomorphic. In each case, it implies that $X$ and $Y$ have the same number of branches, denoted by $n$. We will denote the semibranches of $X$ as $\gamma_0,...,\gamma_{2n-1}$ and the semibranches of $Y$ as $\gamma'_0,...,\gamma'_{2n-1}$.

    Suppose that their H\"older Semicomplexes are isomorphic and let $\sigma\colon sh(X') \to sh(Y')$ be the isomorphism. Consider $\sigma(\gamma_k) = \gamma'_{\tau(k)}$. See that $\tau\colon \Z_n \to \Z_n$ is an isomorphism of associated semicomplexes. 

    If their associated semicomplexes are isomorphic, consider $\tau\colon \Z_n \to \Z_n$ to be the isomorphism. Consider $\sigma\colon sh(X') \to sh(Y')$ as $\sigma(\gamma_k) = \gamma'_{\tau(k)}$ and $\sigma(C(\gamma_k, 0)) = C(\gamma'_{\tau(k)}, 0)$. See that $\sigma$ is an isomorphism of H\"older Semicomplexes. 
\end{proof}
 One may wonder if every semicomplex is associated with a nodal set, but the answer is no. 

\begin{example}
    Consider the semicomplex $(\Z_2, \alpha)$ such that $\alpha(0) = 3$ and $\alpha(1) = 4$. Assume, by contradiction, that $X = f^{-1}(0)$ and $as(X) = (\Z_2, \alpha)$, where $f$ is the real part of 
    
    $$F(z) = iz^2 + (a_3 + ib_3)z^3 + (a_4 + ib_4)z^4 + (a_5 + ib_5)z^5 + ...$$

     As the smallest non vanishing power of $F(z)$ is $2$, $X$ has two branches, each of them parallel to coordinate axis (see Remark \ref{Rem: number of branches equals order}). The linear contact order of the branch parallel to the $x$-axis (in the direction of $e^\frac{0\pi i}{2}$) is $$\inf\{k \in \N : a_k \not = 0\} - 1,$$ and the linear contact order of the branch parallel to the $y$-axis (in the direction of $e^\frac{\pi i}{2}$) is $$\inf\left\{k \in \N : a_k \cos\left(\frac{\pi k}{2}\right) \not = b_k \sin\left(\frac{\pi k}{2}\right)\right\} - 1.$$ Therefore, the linear contact orders with both the horizontal and vertical lines, respectively, are $\alpha(0) = 3$ and $\alpha(1) = 4$ (see Definition \ref{Def: assoc semicomplex}). On one hand,  

    $$\alpha(0) = 3 = \inf\{k \in \N : a_k \not = 0\} - 1 $$ implies that $a_4\not=0$. On the other hand,     
       $$\alpha(1) = 4 = \inf\left\{k \in \N : a_k \cos\left(\frac{\pi k}{2}\right) \not = b_k \sin\left(\frac{\pi k}{2}\right)\right\} -1$$ implies that 

    $$a_4 \cos\left(\frac{4\pi}{2}\right) =b_4 \sin\left(\frac{4\pi}{2}\right).$$

    However, this last equality brings $a_4=0$, a contradiction. This means that no nodal set is associated with the semicomplex $(\Z_2, \alpha)$.
\end{example}

Generalizing the above argument, we can prove the following proposition.

\begin{proposition}\label{Prop: maior ordem de contato}
    Let $X$ be a nodal set with $as(X) = (\Z_n,\alpha)$. Suppose that for some $s\in \Z_n$, we have $\alpha(0)s \equiv s \pmod{n}$. Then $\alpha(s) \le \alpha(0).$
\end{proposition}

\begin{proof}
    Notice that $\alpha(0)=\infty$ makes the thesis automatically true. So we may assume $\alpha(0) \in \mathbb{N}$. Suppose $X = f^{-1}(0)$, $f = \Re F$, and $F(z) = iz^n + \sum\limits_{k > n} (a_k + ib_k) z^k$. By definition, the linear contact order of $\Gamma_0$ is $\alpha(0)$ (see Definition \ref{Def: assoc semicomplex}). Corollary \ref{Cor: formula geral para contato} implies that 
    $$\inf\{k \in \N : a_k \not = 0\} - n + 1 = \alpha(0).$$ 
Thus $a_{\alpha(0) + n - 1} \not = 0$. Suppose, by contradiction, that $\alpha(0) < \alpha(s)$. By Corollary \ref{Cor: formula geral para contato}, this means that 
    $$\alpha(0) + n - 1 < \alpha(s) + n - 1.$$ 
    Hence, 

    $$\alpha(0) + n - 1 < \inf \left\{ m \in \N : a_m \cos\left(\frac{sm\pi}{n}\right) \not = b_m \sin\left(\frac{sm\pi}{n}\right) \right\}.$$

    The infimum above guarantees that

    $$a_{\alpha(0) + n - 1} \cos\left(\frac{s(\alpha(0) + n - 1)\pi}{n}\right) = b_{\alpha(0) + n - 1}\sin\left(\frac{s(\alpha(0)+ n - 1)\pi}{n}\right).$$

    But if $\alpha(0) s \equiv s \pmod{n}$, then $n \mid s(\alpha(0) + n - 1)$. Hence $\sin\left(\frac{s(\alpha(0) + n - 1)\pi}{n}\right) = 0$, and then $a_{\alpha(0)+ n - 1} = 0$, a contradiction.
\end{proof}

\begin{corollary}\label{Cor: generalização da maior ordem de contato}
    Let $X$ be a nodal set with $as(X)=(\Z_n, \alpha)$. If there are $s_0, s \in \Z_n$ such that $\alpha(s_0)s \equiv s\pmod{n}$, then $\alpha(s+s_0) \le \alpha(s_0)$.
\end{corollary} 

\begin{proof}
    Rotate the set $X$ by a factor of $e^{-\frac{s_0\pi i}{n}}$ and call it $X'$. The associated semicomplex of $X'$ is given by $(\Z_n, \alpha')$, where $\alpha'(s) = \alpha(s + s_0)$. Hence $\alpha'(0) = \alpha(s_0)$. By Proposition \ref{Prop: maior ordem de contato} applied to $X'$, it follows that $\alpha'(0)s \equiv s \pmod{n}$, which implies $\alpha'(s) \le \alpha'(0)$. Therefore, in terms of $\alpha$, we obtain that $\alpha
    (s_0)s \equiv s \pmod{n}$ implies $\alpha(s + s_0) \le \alpha(s_0)$.
\end{proof}

Let $X$ be a nodal set with $as(X)=(\Z_n, \alpha)$ and $s_0 \in \Z_n$. The condition $$\alpha(s_0)s \equiv s \pmod{n}$$ is equivalent to 
    $$n \mid s(\alpha(s_0) - 1),$$
which is equivalent to $$\frac{n}{\gcd(n,\alpha(s_0)-1)} \mid s.$$

    Let us denote by $\langle p\rangle$ the cyclic subgroup of $(\Z_n, +)$ generated by $p$. We now define: 
    
    \begin{definition}\label{Def: H's}
        For each $m\in \Z_n$ we define $$H_{m} := \left\langle \frac{n}{\gcd(n,\alpha(m)-1)}\right\rangle,$$

        if $\alpha(m)$ is finite and $H_m = \Z_n$ if $\alpha(m)$ is infinity.
    \end{definition}
    Definition \ref{Def: H's} allows us to reformulate Corollary \ref{Cor: generalização da maior ordem de contato} as item $(1)$ of Corollary \ref{Cor: resto 1 mod n é máxima}.

\begin{corollary}\label{Cor: resto 1 mod n é máxima}
    Let $X$ be a nodal set with $as(X)=(\Z_n, \alpha)$. 
    \begin{enumerate}
        \item For each $s_0 \in \Z_n$, if $s \in s_0 + H_{s_0} =: \{s_0 + h\mid h\in H_{s_0}\}$, then $\alpha(s)\le \alpha(s_0)$.
        \item If there is $k \in \Z_n$ such that $\alpha(k) \equiv 1 \pmod{n}$, then $\alpha(k) = \max\limits_{0 \le j \le n-1}\alpha(j)$.
    \end{enumerate}
\end{corollary}

\begin{proof}
    It only remains to prove $(2)$. If $n \mid \alpha(k) - 1$, then $\frac{n}{\gcd(n,\alpha(k)-1)} = 1$, hence $H_k = \Z_n$ and then $\alpha(s) \le \alpha(k)$ for every $s \in \Z_n$.
\end{proof}

\begin{remark}
    Item $(2)$ of Corollary \ref{Cor: resto 1 mod n é máxima} not only shows how to construct semicomplexes that are not associated with any nodal set, but also shows that if the linear contact order is equivalent to $1 \pmod{n}$, where $n$ is the number of branches, then this contact order is the largest one. 
\end{remark}

\begin{example}\label{Exmp: exemplo do calculo dos H}
    Consider $F(z) = iz^6 + iz^9 + iz^14 + (1 + 2i)z^{15}$ and $X = f^{-1}(0)$ where $f = \Re F$. Let $as(X) = (\Z_6, \alpha)$. Then the values of $\alpha:\Z_6 \to \N$ are given in the following table:

\begin{table}[ht]
\centering
\begin{tabular}{c*{6}{c}}
\\
\toprule
$s$ & $0$ & $1$ & $2$ & $3$ & $4$ & $5$ \\
\midrule
$\alpha(s)$ & $10$ & $4$ & $9$ & $4$ & $9$ & $4$ \\
\bottomrule
\\
\end{tabular}
\caption{Linear contact orders of the branches of $X$} 
\label{tab:example-4.11}
\end{table}

    This example will motivate Theorem \ref{Thm: distribution of the LCO's in nodal sets}, since we will make this result explicit in a particular case.

    Consider $s_1 = 1$ and $s_2 = 2$. Let us determine the subgroups defined in Definition \ref{Def: H's}.  Indeed, $$H_{s_1} = \left\langle \frac{6}{\gcd(6,\alpha(1) - 1)}\right\rangle = \langle 2 \rangle.$$ Similarly, $$H_{s_2} = \left\langle \frac{6}{\gcd(6,\alpha(2) - 1)}\right\rangle = \langle 3 \rangle.$$ 
    
    Notice that $\alpha(s_1) < \alpha(s_2) < \alpha(0)$ and $\Z_6 \supset H_{s_1} \supset H_{s_1} \cap H_{s_2}$. Consequently, 

    $$\alpha(s) = 4  \iff s \in H_{s_1}^C=\{1,3,5\};$$
    $$\alpha(s) = 9 \iff s\in H_{s_1} \cap H_{s_2}^C = \{2,4\};$$
    $$\alpha(s)=  10 \iff s \in H_{s_1} \cap H_{s_2}=\{0\}.$$
\end{example}

\begin{theorem}\label{Thm: distribution of the LCO's in nodal sets}
    Let $X$ be a nodal set with $as(X) = (\Z_n, \alpha)$. Suppose $X$ has exactly $r+1>1$ distinct linear contact orders satisfying 

    $$\alpha(s_1) < \alpha(s_2) <...<\alpha(s_r) < \alpha(0).$$

    Then $\alpha(s) = \alpha(s_k)$ if, and only if, $s \in \bigcap\limits_{j = 1}^{k-1} H_{s_j} \cap H_{s_k}^C$. Moreover, $\alpha(s) = \alpha(0)$ if, and only if, $s \in \bigcap\limits_{j = 1}^{r} H_{s_j}$.
\end{theorem}

\begin{proof}
    First, notice that we may assume $\alpha(0)<\infty$, once the the case $\alpha(0)=\infty$ follows analogously. Second, assume that $X=f^{-1}(0)$ where $f = \Re(F)$ and $F(z) = iz^n + \sum\limits_{k > n} (a_k + ib_k) z^k$.
    Let us define, for each $1\le j\le r$, $$m_j = \alpha(s_j) + n - 1.$$ Notice that $$\alpha(s_1) < \alpha(s_2) <\cdots<\alpha(s_r) < \alpha(0)$$ implies $$m_1<m_2<\cdots < m_r <\alpha(0) + n - 1.$$ 
    
    Using the convention on $\alpha$ which guarantees that $\alpha(0)$ is the linear contact order with the nonnegative part of the $x$-axis allows us to apply Theorem \ref{Teo: formula da ordem de contato}, which implies that we can write $F(z)$ as $$F(z) = iz^n + ib_{m_1}z^{m_1}+...+ib_{m_r}z^{m_r} + (a_{\alpha(0) + n - 1} + ib_{\alpha(0) + n - 1})z^{\alpha(0) + n - 1} + \cdots,$$
    since we can ignore the nonzero $b_k$ for $k > \alpha(0) + n - 1$ when they do not interfere in the contact order formula. 
    By Corollary \ref{Cor: formula geral para contato}, the linear contact orders depend only on the argument of the $k$-th coefficient of the power series of $F$ centered at zero, hence we may assume that each $b_{m_j} = 1$ for each $1\le j\le r$.

    Finally, let us consider the implications of the theorem. Adding $n-1$ to both sides of the equality $\alpha(s) = \alpha(s_k)$, we obtain through Corollary \ref{Cor: formula geral para contato} that it is equivalent to 
    
    $$ \inf\left\{l \in \N :a_l \cos\left( \frac{s l \pi}{n} \right) \not = b_l \sin\left(\frac{s l\pi}{n}\right) \right\} = m_k, $$

    which is equivalent to 

    $$\sin\left(\frac{sm_l\pi}{n}\right) = 0, $$ 

    since $a_{m_l} = 0$ for every $l = 1,2,...,{k-1}$, and 
    
    $$\sin\left(\frac{sm_k\pi}{n}\right) \ne 0.$$
    Moreover, these last conditions are equivalent to $n \mid sm_l = s(\alpha(s_l) + n - 1)$ for every $l=1,2,...,k-1$ and $n \nmid sm_k = s(\alpha(s_k) + n - 1)$. 

    However, recall that $n \mid s(\alpha(s_l) - 1)$ is equivalent to $s \in H_{s_l}$, and $n \nmid s(\alpha(s_k) - 1)$ is equivalent to $s \not \in H_{s_k}$ (see Definition \ref{Def: H's}). Therefore, the statement of the theorem follows. 
\end{proof}

Theorem \ref{Thm: distribution of the LCO's in nodal sets} shows that the distribution of linear contact orders is determined by a filtration of $\Z_n$. 

\begin{definition}
    Let $X$ be a nodal set satisfying the hypotheses of Theorem \ref{Thm: distribution of the LCO's in nodal sets}. Let $H_{s_m}$ as in Definition \ref{Def: H's}, for each $m = 1,2,...r$. We define the \textbf{$m$-th subgroup of contact of $X$} to be the set $K_1 = \Z_n$ if $m=1$ or

    $$K_m = \bigcap_{k = 1}^{m-1} H_{s_k}$$

    for $m = 2,...,r+1$. 

    The filtration 

    $$\Z_n = K_1 \supsetneq K_2 \supsetneq\cdots \supsetneq K_r \supsetneq K_{r+1}$$

    is called the \textbf{contact filtration of $X$}. 
\end{definition}

\begin{remark}\label{Rem: Teorema principal reescrito em filtracoes}
    Theorem \ref{Thm: distribution of the LCO's in nodal sets} shows that $\alpha$ is constant along $K_{r+1}$ and $K_m \setminus K_{m+1}$ for each $m=1,\ldots, r$. For different $m$, this constant valuation is different, then the values of $\alpha$ are determined by its contact filtration. 
\end{remark}

\begin{example}
    Recall Example \ref{Exmp: exemplo do calculo dos H}, let us see how Theorem \ref{Thm: distribution of the LCO's in nodal sets} can be used to simplify the calculation of $\alpha$ if we already know that $\alpha(0)$ is the largest contact order. 

    Since each $H_{s_j}$ is a subgroup of $\Z_n$, it is either $\Z_n$ or generated by a divisor of $n$. Then we need only to calculate $\alpha(d)$ for any $d$ such that $d \mid n$. 

    If $F(z) = iz^6 + iz^9 + iz^14 + (1 + 2i)z^{15}$, $n = 6$, then we need only to calculate $\alpha(1) = 4$, $\alpha(2) = 9$, and $\alpha(3) = 4$.
    
    By Example \ref{Exmp: exemplo do calculo dos H}, we know that $$H_1 = \langle 2\rangle = \{0, 2, 4\}$$ and $$H_2 = \langle 3 \rangle = \{0, 3\}.$$ Therefore, $$K_1 = \Z_6 \supsetneq K_2 = H_1 \supsetneq K_3 = H_1 \cap H_2 = \{0\}.$$ By Theorem \ref{Thm: distribution of the LCO's in nodal sets}, $\alpha(s) = \alpha(1)$ if, and only if,
    $$s \in K_1 \setminus K_2 = \{0,1,2,3,4,5\} \setminus \{0,2,4\} = \{1,3,5\}.$$ 
    Then $\alpha(1) = \alpha(3) = \alpha(5) = 4$. We also have $\alpha(s) = \alpha(2) = 9$ if, and only if, 
    $$s \in K_2 \setminus K_3 = \{0,2, 4\} \setminus \{0\} = \{2,4\},$$ 
which implies that $\alpha(2) = \alpha(4) = 9$. Since $K_3 = \{0\}$, $\alpha(0) = 10$ is the unique maximum linear contact order. 

    Generalizing this algorithm, one can easily calculate all the linear contact orders of any nodal set. 
\end{example}

\begin{corollary}
    Let $X$ be a nodal set with $as(X)=(\Z_n, \alpha)$. For each $s \in \Z_n$, we have $$\min \{\alpha(s), \alpha(s+1)\}= \min\limits_{0 \le j \le n - 1} \alpha(j).$$ In words, the minimal linear order necessarily occurs in each pair of adjacent semibranches.
\end{corollary}

\begin{proof}
    Let us start by assuming, without loss of generality, that we arranged the distinct linear contact orders of $X$ so that  $\alpha(s_1)<\alpha(s_2)<...<\alpha(s_r)<\alpha(0)$. Notice that $H_{s_1} \ne \Z_n$. Indeed, otherwise Theorem \ref{Thm: distribution of the LCO's in nodal sets} implies that $\alpha(s) = \alpha(s_1)$ if, and only if, $s \in H_{s_1}^C = \Z_n^C$, a contradiction with $s=s_1\in \Z_n$. 

Moreover, for each $s \in \Z_n$, we cannot have $s, s+1\in H_{s_1}$, otherwise, $1=(s+1) - s \in H_{s_1}$, but $H_{s_1} \ne \Z_n$. Therefore, for each $s \in \Z_n$, one of the numbers $s, s+1$ is in $H_{s_1}^C$, which implies that $\alpha(s)$ or $\alpha(s+1)$ is $\alpha(s_1)$, which is the least contact order. 
\end{proof}

\begin{corollary}
    Let $X$ be a nodal set with $as(X) = (\Z_n, \alpha)$. Given a positive integer $n$, let $\Omega(n) = e_1 + \cdots + e_r$, where $n = \prod_{j=1}^rp_j^{e_j}$ and $p_1,\ldots, p_r$ are distinct prime numbers.  Then $1+\Omega(n)$ is an upper bound on the number of distinct linear contact orders, i.e., $$|\text{Im}(\alpha)|\le 1 + \Omega(n).$$
\end{corollary}

\begin{proof}
    Suppose that $\alpha(s_1)<\alpha(s_2)<...<\alpha(s_k)<\alpha(0)$ are the linear contact orders of $X$. Our aim is to bound the number $k$. By Theorem \ref{Thm: distribution of the LCO's in nodal sets}, there are subgroups $H_{s_1},H_{s_2},...,H_{s_k}$ of $\Z_n$ generated by divisors of $n$ such that for each $l = 1, 2,...,k$, the set $$\bigcap\limits_{j=1}^{l-1} H_{s_j} \cap H_l^C$$ is nonempty, once at least $s_l$ belongs to it. We shall prove that $\Omega(n)$ is the maximum number of such subgroups, i. e., $k\le \Omega(n)$. Consider the chain of subgroups given by
    $$\Z_n \supseteq H_{s_1} \supseteq H_{s_1}\cap H_{s_2} \supseteq \cdots \supseteq \bigcap_{l = 1}^k H_{s_l}$$
    Each subgroup of this chain is generated by a divisor of $n$. Let $d_i$ be the generator of $$K_i = \bigcap_{l = 1}^{i-1} H_{s_l}.$$ The inclusions guarantee that $d_i \mid d_{i+1}$ for each $1\le i\le k-1$. Since $$\bigcap\limits_{j=1}^{l-1} H_{s_j} \cap H_l^C$$
    is nonempty for each $l$, we have $K_i \supsetneq K_{i+1}$ for each $i$. Nevertheless, this implies that $d_{i}<d_{i+1}$. Hence $d_{i+1}$ has at least one more prime number in its decomposition into a product of primes. Therefore, the strategy to maximize $k$ is, when constructing the divisors $d_1,\ldots,d_k$, from $d_i$ to $d_{i+1}$, to add exactly one prime number to the expansion of $d_{i+1}$, but we can add at most $\Omega(n)$ prime numbers. 
\end{proof}

\begin{theorem}\label{Thm: realization theorem}
    Let $(\Z_n, \alpha)$ be a semicomplex such that $$Im(\alpha) = \{\alpha(s_1),\ldots,\alpha(s_r),\alpha(0)\}$$ and

    $$\alpha(s_1) < \alpha(s_2) <...< \alpha(s_r) < \alpha(0).$$

    Moreover, for each $m = 1,2,...,r$, let $L_m = \frac{n}{\gcd(n, \alpha(s_m) - 1)}$ and consider the sets $$K'_1 = \left\langle L_1\right\rangle^C , \hspace{20pt}  K'_m = \bigcap_{k = 1}^{m - 1} \left\langle L_k\right\rangle \cap \left\langle L_m\right\rangle^C, \; m = 2,...,r, \hspace{20pt} K'_{r+1} = \bigcap_{k = 1}^{r} \left\langle L_k\right\rangle,$$ and assume that $s_m \in K'_m$ for each $m = 1,2,...,r$.
    
    If $\alpha$ is constant along each of the sets $K'_1,\ldots,K'_{r+1}$, then there is a nodal set $X$ such that $as(X) = (\Z_n, \alpha)$. 
\end{theorem}

\begin{proof}
Notice that $s_m \in K'_m$ implies that $s_m \in \langle \frac{n}{\gcd(n, \alpha(s_j) - 1)}\rangle$, for $1\le j < m$. Hence $\frac{n}{\gcd(n, \alpha(s_j) - 1)} \mid s_m$, which implies that $n \mid s_m(\alpha(s_j) - 1)$ for every $1\le j < m$. Similarly, $s_m \not \in \langle \frac{n}{\gcd(n, \alpha(s_m) - 1)}\rangle$ implies that $n \nmid s_m(\alpha(s_m) - 1)$. 

In what follows, consider the function $f(z) = \Re F(z)$ and the nodal set $X = f^{-1}(0)$ with $as(X) = (\Z_n, \beta)$. We will choose a suitable $F(z)$ for each case.

    If $\alpha(0) = \infty$, consider $$F(z) = iz^n + iz^{\alpha(s_1) + n - 1} +iz^{\alpha(s_2) + n - 1}+\cdots+iz^{\alpha(s_r) + n - 1}.$$ We shall prove that $\beta = \alpha$. By Corollary \ref{Cor: formula geral para contato}, $$\beta(s_m) + n - 1 = \inf\left\{k \in \N : a_k\cos\left(\frac{s_mk\pi}{n}\right) \ne b_k\sin\left(\frac{s_mk\pi}{n}\right)\right\}.$$ Thus, since $a_k=0$ for any $k$, $\alpha(s_m) +n - 1 = \beta(s_m) + n - 1$ is equivalent to 
    $$\sin\left(\frac{s_m(\alpha(s_m) + n - 1)\pi}{n}\right) \ne 0,$$

    $$\sin\left(\frac{s_m(\alpha(s_j)+n-1)\pi}{n}\right) = 0,$$

    for $j<m$. Once $n \mid s_m(\alpha(s_j) - 1)$ for $1\le j < m$ and $n \nmid s_m(\alpha(s_m) - 1)$, then $\beta(s_m) = \alpha(s_m)$. By Theorem \ref{Thm: distribution of the LCO's in nodal sets}, the function $\beta$ is constant in each $K'_m$, since $\beta(s) = \beta(s_m) = \alpha(s_m)$ if, and only if, $s \in \bigcap\limits_{k = 1}^{m - 1} H_{s_j} \cap H_m^C$. By definition, $H_{s_j} = \langle \frac{n}{\gcd(n, \beta(s_j) - 1)}\rangle = \langle \frac{n}{\gcd(n, \alpha(s_j) - 1)}\rangle = \langle L_j\rangle$. Therefore, $\beta(s) = \beta(s_m)$ if, and only if $s \in \bigcap\limits_{k = 1}^{m - 1} \langle L_k\rangle  \cap \langle L_m\rangle ^C = K'_m$. Hence $\beta$ is constant over each $K'_m$, for $m = 1,2,...,r$ and $\beta(s) = \infty$ if, and only if $s \in K'_{r+1}$, which implies that $\beta = \alpha$. 

    If $\alpha(0) < \infty$, consider $$F(z) = iz^n + iz^{\alpha(s_1) + n - 1} +iz^{\alpha(s_2) + n - 1}+\cdots+iz^{\alpha(s_r) + n - 1}$$
    \vspace{-1cm}
    $$+(a_{\alpha(0) + n - 1} + ib_{\alpha(0) + n - 1})z^{\alpha(0) + n - 1}.$$ By Corollary \ref{Cor: formula geral para contato}, the linear contact orders of $X$ depend only on the angle of the coefficients of the power series of $F$, hence we may assume that $a_{\alpha(0) + n - 1} = 1$, since $a_{\alpha(0) + n - 1} \ne 0$. 

    For each $s \in K'_{r+1}$, let $Y_s$ be the set of $y \in \R$ such that the equality $$\cos\left(\frac{s(\alpha(0) + n - 1)\pi}{n}\right) = y\sin\left(\frac{s(\alpha(0) + n - 1)\pi}{n}\right)$$ holds. Each $Y_s$ is at most countable, since the trigonometric function $\cot$ is periodic. This implies that $\bigcup_{s=1}^rY_s$ is countable. Let us choose $b = b_{\alpha(0) + n - 1} \in \bigcap_{s=1}^rY_s^C$. Let us show that $\beta = \alpha$. 

    See that, as before, $\beta(s_m) = \alpha(s_m)$ and $\beta$ is constant in each $K_m$ for $m=1,2,...,r$. It remains to be proved that $\beta(0) = \alpha(0)$ and that $\beta$ is constant in $K_{r+1}$.

    By Corollary \ref{Cor: formula geral para contato}, $\beta(0) = \inf\{ k \in \N : a_k \ne 0\} - n + 1$, since this infimum is $\alpha(0) + n - 1$, we have proved that $\beta(0) = \alpha(0)$.

    Let $s \in K_{r+1}$, since $b$ is chosen so that 
    $$\cos\left(\frac{s(\alpha(0) + n - 1) \pi}{n}\right) \ne b\sin\left(\frac{s(\alpha(0) + n - 1)\pi}{n}\right)$$

    and $s \in \bigcap\limits_{k = 1}^r \langle \frac{n}{\gcd(n, \alpha(k) - 1)}\rangle$ means that $n \mid s(\alpha(k) - 1)$, then 

    $$\sin\left(\frac{s(\alpha(k)+ n - 1)\pi}{n}\right) = 0$$

    for every $k=1,2,...,r$, which implies that $\beta(s)= \alpha(0)$ for $s \in K_{r+1}$. Therefore, $\beta = \alpha$. And then, in each case, $as(X) = (\Z_n, \alpha)$. 
\end{proof}

\begin{corollary}
    Given an integer $n\ge 1$, there is a nodal set $X$ with $as(X) =(\Z_n, \alpha)$ such that $|\text{Im}(\alpha)| = 1 + \Omega(n)$. In other words, $1 + \Omega(n)$ is a sharp bound for the number of possible contact orders of nodal sets.
\end{corollary}

\begin{proof}
    Let $n = p_1^{r_1}...p_k^{r_k}$ be the prime decomposition of $n$, where $r_j\ge 1$ for each $j=1,\ldots,k$. Let us consider the semicomplex $(\Z_n,\alpha)$ with $Im(\alpha) = \{\alpha(s_1),\ldots,\alpha(s_k), \alpha(0)\}$, where 
    $$\begin{cases}s_1 = 1;\\ s_2 = p_{1}^{2-1};\\\vdots\\s_{r_1+1} = p_1^{r_1};\\s_{r_1 + 2} = p_1^{r_1}p_2^{2-1};\\\vdots\\s_{r_1 + r_2 +...+r_k}=p_1^{r_1}...p_k^{r_k-1}\end{cases}
    $$ 
    and 
    $$\begin{cases}\alpha(s_1) = p_1^{r_1 - 1}p_2^{r_2}...p_k^{r_k}(n^{l_1} + 1)+1;\\\alpha(s_2) = p_1^{r_1 - 2}p_2^{r_2}...p_k^{r_k}(n^{l_2} + 1) +1;\\\vdots\\\alpha(s_{r_1}) = p_2^{r_2}...p_k^{r_k}(n^{l_{r_1}} + 1)+1;\\\alpha(s_{r_1 + 1}) = p_2^{r_2 - 1}p_3^{r_3}...p_k^{r_k}(n^{l_{r_1+1}} + 1)+1;\\\vdots\\\alpha(s_{r_1+...+r_k}) = n^{l_{r_1+...+r_k}}+2\end{cases},
    $$

    where the numbers $l_j \ge 1$ are chosen so that $\alpha(s_1) < \alpha(s_2)<...<\alpha(s_{r_1+...+r_k})$. We finally set $\alpha(0) = \alpha(s_{r_1+...+r_k})+1$.

    Let $K'_m$, $m=1,\ldots,\Omega(n)=r_1+\cdots+r_k$ be as in the hypothesis of Theorem \ref{Thm: realization theorem}. We shall prove that $s_m \in K'_m$ for each $m=1,\ldots,\Omega(n)$. This is equivalent to showing that for every $j < m$, $\frac{n}{\gcd(n, \alpha(s_j) - 1)} \mid s_m$, and $\frac{n}{\gcd(n, \alpha(s_m) - 1)} \nmid s_m$, which are equivalent to $n \mid s_m(\alpha(s_j) - 1)$ and $n \nmid s_m(\alpha(s_m) - 1)$. Notice that $s_j$ and $\alpha(s_j)$ were chosen so that the conditions are satisfied. Therefore, by Theorem \ref{Thm: realization theorem}, there is a nodal set $X$ such that $as(X) = (\Z_n, \alpha)$. Hence, $|\text{Im}(\alpha)| = 1 + r_1 +...+ r_k = 1 + \Omega(n)$. 
\end{proof}

\begin{definition}\label{Def: filtration semicomplex}
    Let $(\Z_n, \alpha)$ be a semicomplex such that $$Im(\alpha) = \{\alpha(s_1),\ldots,\alpha(s_r),\alpha(0)\}.$$ We say that $(\Z_n, \alpha)$ is a \textbf{filtration semicomplex} if it satisfies the conditions in the hypothesis of Theorem \ref{Thm: realization theorem}.
\end{definition}

\begin{remark}\label{Rem: complet alg invariant}
    The associated semicomplex of a nodal set is always a filtration semicomplex, by Theorem \ref{Thm: distribution of the LCO's in nodal sets}. Therefore, Theorem \ref{Thm: realization theorem} implies that the filtration semicomplex is a complete algebraic invariant with respect to the outer Lipschitz Geometry of nodal sets united with their tangent cones. 
\end{remark}

\begin{theorem}
    Let $S =(\Z_n, \alpha)$ a semicomplex. $S$ is a nodal semicomplex if and only if $S$ is a filtration semicomplex. 
\end{theorem}

\subsection{Contact Orders of Unions of Nodal Sets} 

Along this subsection we will now study the union of finitely many nodal sets $Z(f_1), \ldots, Z(f_p)$ such that $$0\in \bigcap_{i=1}^{p}Z(f_i).$$ In which the opposite semibranch of $\Gamma$ in $X$ is its opposite semibranch in $Z(f_i)$, for some $i$ such that $\Gamma$ is a semibranch of $Z(f_i)$.

\begin{lemma}\label{Lem: image of nodal set is nodal set} 
    Let $f\colon B(0,\varepsilon) \to \R$ be a harmonic function, $X = Z(f)$ and $h\colon B(0,\varepsilon) \to B(0,\delta)$ be a holomorphic homeomorphism. Then $h(X)$ is a nodal set. 
\end{lemma}

\begin{proof}
By Remark \ref{Rem: harmonic is locally real analytic}, there is an analytic function $F\colon B(0,\varepsilon) \to \C$ such that $f(z) = \Re(F(z))$. 

    Hence, $X = Z(f) = f^{-1}(0) = F^{-1}(i\R)$, and then $$h(X) = h(F^{-1}(i\R)) = (F \circ h^{-1})^{-1}(i\R).$$ Thus, as $F \circ h^{-1}$ is an analytic function, this implies that the zero level of $g(z) = \Re(F(h^{-1}(z))) $ is exactly $h(X)$. 
\end{proof}

\begin{definition}
    Let $\alpha\colon \Z_n \to \N_{\infty} $ be a (resp., nodal or filtration) semicomplex. We call a restriction of $\alpha$ to a coset of a subgroup of $\Z_n$ a \textbf{partial (resp., nodal or filtration) semicomplex}. An \textbf{isomorphism of partial semicomplexes} $\beta_i\colon K_i \to \N_{\infty}$ for $i = 1,2$ is a bijection $\gamma\colon K_1 \to K_2$ such that $\beta_1(k)= \beta_2(\gamma(k))$ for each $k \in K_1$.
\end{definition}

Although a formula for the contact order between two branches of distinct nodal sets is extremely complex, by Lemma \ref{Lem: image of nodal set is nodal set}, one can generalize Theorem \ref{Thm: distribution of the LCO's in nodal sets} to unions of two nodal sets, since by Lemma \ref{Lemma: Lema da Topologia de conjuntos nodais}, there is an analytic bijection between one nodal set and its tangent cone which transforms nodal sets into nodal sets by Lemma \ref{Lem: image of nodal set is nodal set}, preserving contact orders by Lemma \ref{Lem: holomorohic preserves}. Therefore, the general case is reduced to Theorem \ref{Thm: distribution of the LCO's in nodal sets}. Theorem \ref{Thm: generalização do principal} is a precise formulation of the result discussed in this context.

Consider germs at $0$ of nodal sets $(X,0)$ and $(Y,0)$ with $n$ and $m$
branches, respectively, having at least two tangent branches. Without loss of generality, we may assume that the
$x$-axis is tangent to a branch of $X$ and to a branch of $Y$. We denote by
$\Gamma_X^0,\ldots,\Gamma_X^{n-1}$ and $\Gamma_Y^0,\ldots,\Gamma_Y^{m-1}$ the
branches of $X$ and $Y$, respectively, numbered counterclockwise from the
positive $x$-axis. Thus, by
Lemma~\ref{Lemma: Lema da Topologia de conjuntos nodais}, $\Gamma_X^k$ is
tangent to the line $\mathbb{R}e^{\frac{k\pi i}{n}}$ and $\Gamma_Y^l$ is
tangent to the line $\mathbb{R}e^{\frac{l\pi i}{m}}$, for $k\in\mathbb{Z}_n$
and $l\in\mathbb{Z}_m$. We write $as(X)=(\mathbb{Z}_n,\alpha_X)$. 

\begin{theorem}\label{Thm: generalização do principal}
Let $(X,0)$ and $(Y,0)$ be as above, and let
$\Gamma_X^{k_1},\ldots,\Gamma_X^{k_r}$ and $\Gamma_Y^{l_1},\ldots,\Gamma_Y^{l_r}$
be the only branches such that $\mathrm{Cont}(\Gamma_Y^{l_s},\Gamma_X^{k_s})>1$
for $s=1,\ldots,r$. Then:
\begin{enumerate}
  \item $r=\gcd(n,m)$;
  \item $K=\{k_1,\ldots,k_r\}$ is the subgroup
  $\left\langle \frac{n}{\gcd(n,m)}\right\rangle$ of $\mathbb{Z}_n$;
  \item the function $\beta_X\colon K\to\mathbb{N}_\infty$ given by
  $\beta_X(k_s)=\mathrm{Cont}(\Gamma_Y^{l_s},\Gamma_X^{k_s})$, $s=1,\ldots,r$,
  is a partial nodal semicomplex.
\end{enumerate}
\end{theorem}

\begin{proof}
Let $d=\gcd(n,m)$ and write $n=dn'$ and $m=dm'$, so that $\gcd(n',m')=1$.

By Lemma~\ref{Lemma: Lema da Topologia de conjuntos nodais}, distinct branches
of a nodal set have distinct tangent cones, and by Definition~\ref{Def: angle},
two branches satisfy $\mathrm{Cont}>1$ if and only if they have the same
tangent cone. Hence, each branch of $X$ is tangent to at most one branch of
$Y$ and conversely. Therefore $r=|K|$, and $K$ is the set of
$k\in\mathbb{Z}_n$ such that $\Gamma_X^k$ and some $\Gamma_Y^l$ have the same
tangent line. Since two vectors $e^{ia}$ and $e^{ib}$ are colinear if
and only if $a-b\in\pi\mathbb{Z}$, we obtain that $k\in K$ if and only if
there is $l\in\{0,\ldots,m-1\}$ such that
\begin{equation}\label{eq:common line}
\frac{k}{n}-\frac{l}{m}\in\mathbb{Z}.
\end{equation}
This condition does not depend on the representatives of $k$ modulo $n$ and
of $l$ modulo $m$, since replacing $k$ by $k+n$ or $l$ by $l+m$ changes the
left-hand side by an integer.

\textit{(1) and (2).} Let us show that $k\in K$ if and only if $n'\mid k$. This
condition is well defined on $\mathbb{Z}_n$, since $n'\mid n$.

Suppose that \eqref{eq:common line} holds, say
$\frac{k}{n}-\frac{l}{m}=q\in\mathbb{Z}$. Multiplying by $dn'm'$ gives
\[
km'-ln'=q\,d\,n'm',
\]
then $n'\mid km'$. Since $\gcd(n',m')=1$, it follows that $n'\mid k$.

Conversely, suppose that $k=jn'$ for some $j\in\mathbb{Z}$, and let
$l\in\{0,\ldots,m-1\}$ be such that $l\equiv jm'\pmod m$. Then
\[
\frac{jn'}{dn'}-\frac{jm'}{dm'}=0,
\]
therefore, by the remark after \eqref{eq:common line}, the pair $(k,l)$ satisfies
\eqref{eq:common line}. 

Since $k\in K$ if and only if $n'\mid k$, we obtain $K=\{k\in\mathbb{Z}_n : n'\mid k\}=\langle n'\rangle
=\left\langle \frac{n}{d}\right\rangle$, which proves (2). Since
$\langle n'\rangle$ is the subgroup of $\mathbb{Z}_n$ of order $n/n'=d$, we
get $r=|K|=d$, which proves (1).

\textit{(3).} By Lemma \ref{Lemma: Lema da Topologia de conjuntos nodais}, there is a
holomorphic diffeomorphism
\[
h\colon V\to\mathbb{C},\qquad
\]

for $V$ a neighborhood of the origin, such that we have
\begin{enumerate}
  \item[(a)] for every $s=1,\ldots,r$, $h(\Gamma_Y^{l_s})$ is, as a germ at
  $0$, the tangent line of $\Gamma_Y^{l_s}$, which is
  $\mathbb{R}e^{\frac{k_s\pi i}{n}}$;
  \item[(b)] for every $k\in\mathbb{Z}_n$, $h(\Gamma_X^k)$ is a branch tangent
  to $\mathbb{R}e^{\frac{k\pi i}{n}}$.
\end{enumerate}
By Lemma~\ref{Lem: image of nodal set is nodal set}, $h(X\cap V)$ is a nodal
set, and by (b) its branches are $h(\Gamma_X^0),\ldots,h(\Gamma_X^{n-1})$,
numbered counterclockwise from the positive $x$-axis. Thus we write
$as(h(X))=(\mathbb{Z}_n,\alpha_{h(X)})$ with
$\alpha_{h(X)}(k)=LC(h(\Gamma_X^k),0)$.

Since $h'(0)\neq0$, Lemma~\ref{Lem: holomorohic preserves} and (a) give, for
every $s=1,\ldots,r$,
\[
\beta_X(k_s)
=\mathrm{Cont}\bigl(\Gamma_Y^{l_s},\Gamma_X^{k_s}\bigr)
=\mathrm{Cont}\bigl(h(\Gamma_Y^{l_s}),h(\Gamma_X^{k_s})\bigr)
= \] \[
= \mathrm{Cont}\bigl(\mathbb{R}e^{\frac{k_s\pi i}{n}},h(\Gamma_X^{k_s})\bigr)
=LC\bigl(h(\Gamma_X^{k_s}),0\bigr)
=\alpha_{h(X)}(k_s),
\]
where the fourth equality follows from Definitions~\ref{Def: linear contact order}
and~\ref{Def: branch cont order}. Therefore $\beta_X=\alpha_{h(X)}|_K$. Since
$(\mathbb{Z}_n,\alpha_{h(X)})$ is the associated semicomplex of a nodal set
and, by (2), $K$ is a subgroup of $\mathbb{Z}_n$, it follows that $\beta_X$ is a partial nodal semicomplex.
\end{proof}

\begin{corollary}\label{cor: group filtration for two nodal sets}
    In the hypotheses of Theorem \ref{Thm: generalização do principal}, there is a group filtration 

    $$\Z_n = K_1  \supsetneq K_2\supsetneq K_3\supsetneq \dots \supsetneq K_r$$

    such that $\beta$ is constant along $K\cap (K_s \setminus K_{s+1})$ for each $s = 0,...,r-1$. We also obtain that the function $s \mapsto \beta(K \cap (K_s \setminus K_{s+1}))$ is increasing. 
\end{corollary}




\begin{corollary}\label{Cor: realization of partials}
Let $H$ be a subgroup of $\mathbb{Z}_n$ and let $\beta'\colon H\to\mathbb{N}_\infty$
be a partial filtration semicomplex, that is, $\beta'=\alpha|_H$ for some
filtration semicomplex $(\mathbb{Z}_n,\alpha)$. Then there are germs of nodal
sets $(X,0)$ and $(Y,0)$ satisfying the hypotheses of
Theorem~\ref{Thm: generalização do principal} such that $K=H$ and the function
$\beta_X$ of Theorem~\ref{Thm: generalização do principal} is exactly $\beta'$.
\end{corollary}

\begin{proof}
By Theorem~\ref{Thm: realization theorem}, there is a germ of a nodal set $(X,0)$ such that
$as(X)=(\mathbb{Z}_n,\alpha)$. That is, the semibranches
$\Gamma_X^0,\ldots,\Gamma_X^{n-1}$ of $X$ are numbered counterclockwise from
the positive $x$-axis, $\Gamma_X^k$ is tangent to the line
$\mathbb{R}e^{\frac{k\pi i}{n}}$, and $LC(\Gamma_X^k,0)=\alpha(k)$ for every
$k\in\mathbb{Z}_n$.

Let $h=|H|$. Since $\mathbb{Z}_n$ is cyclic, $h\mid n$ and $H$ is its unique
subgroup of order $h$, namely $H=\left\langle \frac{n}{h}\right\rangle
=\left\{ \frac{jn}{h} : j=0,\ldots,h-1\right\}$. Consider the harmonic function
$g(x,y)=\mathfrak{Re}(i z^h)=-\mathfrak{Im}(z^h)$ and $Y=Z(g)$. Then $z\in Y$ if
and only if $z^h\in\mathbb{R}$, hence
\[
Y=\bigcup_{j=0}^{h-1}\mathbb{R}e^{\frac{j\pi i}{h}}
\]
is a nodal set whose branches are the $h$ lines
$\Gamma_Y^j=\mathbb{R}e^{\frac{j\pi i}{h}}$, $j=0,\ldots,h-1$, numbered
counterclockwise from the positive $x$-axis. Therefore, $X$ and $Y$ satisfy the preceding normalization condition of Theorem~\ref{Thm: generalização do principal}.

Since $e^{\frac{j\pi i}{h}}=e^{\frac{(jn/h)\pi i}{n}}$, the line $\Gamma_Y^j$
is the tangent line of $\Gamma_X^{jn/h}$. Since the lines
$\mathbb{R}e^{\frac{k\pi i}{n}}$, $k\in\mathbb{Z}_n$, are pairwise distinct,
a branch $\Gamma_X^k$ has the same tangent line as some branch of $Y$ if and
only if $k\in H$, and in that case this branch of $Y$ is unique. By
Definition~\ref{Def: angle}, the branches of $X$ and $Y$ with contact order
greater than $1$ are exactly the pairs $\Gamma_X^{jn/h}$, $\Gamma_Y^j$, for
$j=0,\ldots,h-1$. Therefore, in Theorem~\ref{Thm: generalização do principal},
we have $r=h$ and $K=H$.

Finally, let $k=\frac{jn}{h}\in H$. The branch $\Gamma_Y^j$ is the tangent
line of $\Gamma_X^k$, so each semibranch of $\Gamma_Y^j$ is the tangent cone
at $0$ of the semibranch of $\Gamma_X^k$ lying on the same side. By
Definitions~\ref{Def: linear contact order},
\[
\beta_X(k)=\mathrm{Cont}\bigl(\Gamma_Y^j,\Gamma_X^k\bigr)=LC(\Gamma_X^k,0)
=\alpha(k)=\beta'(k).
\]
Hence $\beta_X=\beta'$.
\end{proof}

\begin{remark}
If $\beta'$ is defined on a coset $K=k_0+H$ with $k_0\notin H$, then, by
Theorem~\ref{Thm: generalização do principal}, $\beta'$ cannot arise from a
pair $(X,Y)$ normalized as above, since in that case $K$ is a subgroup.
Nevertheless, replacing $X$ by its rotation $e^{-\frac{k_0\pi i}{n}}X$, whose
associated semicomplex is $s\mapsto\alpha(s+k_0)$, the proof above realizes the
translated function $s\mapsto\beta'(s+k_0)$ on $H$.
\end{remark}

The generalization of Corollaries \ref{Cor: realization of partials}, \ref{cor: group filtration for two nodal sets}, and Theorem \ref{Thm: generalização do principal} to the union of a finite number of curves is not true, as one may think if we extend naturally the notion of partial semicomplexes and the notion of isomorphism of partial semicomplexes. 

\begin{definition}
    A \textbf{complex} is a complete graph $G_n$ with a family of functions that, for any two distinct vertices $v_1, v_2$, we associate a partial semicomplex $\alpha_{v_1,v_2}\colon K_{v_1,v_2} \to \N_{\infty}$. 
\end{definition}

\begin{definition}
    Let $(X_1,0),...,(X_n,0)$ germs of nodal curves. The \textbf{nodal complex} of $X_1\cup \dots \cup X_n$ at $0$ is denoted by $nc(X_1 \cup \dots \cup X_n, 0)$ and is a complete graph whose set of vertices is $\{X_1,...,X_n\}$. For each $i \neq j$, the function associated with $X_i$ and $X_j$ is the function $\beta_{ij}$ given in Theorem \ref{Thm: generalização do principal}. 
\end{definition}

\begin{definition}
    Let $(K_n, \{\alpha_{v_i, v_j} \} _{1 \le i \neq j \le n})$ and $(G_m, \{\beta_{w_i,w_j} \} _{1 \le i \neq j \le m})$ be two complexes. An \textbf{isomorphism of complexes} is a graph isomorphism $\gamma\colon K_n \to G_m$ such that $\alpha_{v_i, v_j}$ and $\beta_{\gamma(v_i), \gamma(v_j)}$ are isomorphic as partial nodal semicomplexes. 
\end{definition}

\begin{example}
Let $X = \bigcup_{i=1}^{4} X_i$ and $Y = \bigcup_{i=1}^{4} Y_i$, where
\begin{align*}
X_1 &\colon\ -2xy - 3x^2y + y^3 = 0, \\
X_2 &\colon\ -2xy - 6x^2y + 2y^3 + x^4 - 6x^2y^2 + y^4 = 0, \\
X_3 &\colon\ -2xy + x^3 - 3xy^2 = 0, \\
X_4 &\colon\ -2xy + 2x^3 - 6xy^2 + x^4 - 6x^2y^2 + y^4 = 0,
\end{align*}
and
\begin{align*}
Y_1 &= X_1, \qquad Y_2 = X_2, \\
Y_3 &\colon\ -2xy + x^3 - 9x^2y - 3xy^2 + 3y^3 = 0, \\
Y_4 &\colon\ -2xy + x^3 - 12x^2y - 3xy^2 + 4y^3 = 0.
\end{align*}
\end{example}

Each $X_i$ and each $Y_i$ has exactly two branches at the origin, one tangent
to the $x$-axis and one tangent to the $y$-axis, since the lowest-degree term
of every defining polynomial is $-2xy = \mathfrak{Re}(iz^2)$. Identify the
direction of the $x$-axis with $0 \in \mathbb{Z}_2$ and the direction of the
$y$-axis with $1 \in \mathbb{Z}_2$. Then, for every $i \neq j$, the function
of Theorem~\ref{Thm: generalização do principal} associated with $X_i$ and $X_j$ is defined on all of
$\mathbb{Z}_2$, and so is the one associated with $Y_i$ and $Y_j$. We denote
them by $\beta^X_{ij}, \beta^Y_{ij} \colon \mathbb{Z}_2 \to \mathbb{N}_\infty$.

\begin{table}[ht]
\centering
\begin{tabular}{ccc@{\hskip 2em}ccc}
\\
\toprule
Branches & $x$-axis & $y$-axis & Branches & $x$-axis & $y$-axis \\
\midrule
$\beta_{1,2}^X$ & $\mathbf{3}$ & $2$ & $\beta_{1,2}^Y$ & $\mathbf{3}$ & $2$ \\
$\beta_{1,3}^X$ & $2$ & $2$ & $\beta_{1,3}^Y$ & $2$ & $2$ \\
$\beta_{1,4}^X$ & $2$ & $2$ & $\beta_{1,4}^Y$ & $2$ & $2$ \\
$\beta_{2,3}^X$ & $2$ & $2$ & $\beta_{2,3}^Y$ & $2$ & $2$ \\
$\beta_{2,4}^X$ & $2$ & $2$ & $\beta_{2,4}^Y$ & $2$ & $2$ \\
$\beta_{3,4}^X$ & $2$ & $\mathbf{3}$ & $\beta_{3,4}^Y$ & $\mathbf{3}$ & $2$ \\
\bottomrule
\\
\end{tabular}
\caption{Contact orders of the tangent branches with tangent cone in the direction of the $x$-axis and $y$-axis.}
\label{tab:contatos}
\end{table}

We claim that $nc(X_1 \cup \dots \cup X_4, 0)$ and $nc(Y_1 \cup \dots \cup Y_4, 0)$
are isomorphic. Let
$\gamma \colon K_4 \to K_4$ be the graph isomorphism $\gamma(X_i) = Y_i$. We
must show that, for each pair $i \neq j$, the partial nodal semicomplexes
$\beta^X_{ij}$ and $\beta^Y_{ij}$ are isomorphic, that is, that there is a
bijection $\gamma_{ij} \colon \mathbb{Z}_2 \to \mathbb{Z}_2$ with
$\beta^Y_{ij}(\gamma_{ij}(k)) = \beta^X_{ij}(k)$ for $k \in \mathbb{Z}_2$, for each $i \neq j \in \{1,2, 3,4\}$
\begin{itemize}
  \item For $\{i,j\} \neq \{3,4\}$ we have $\beta^X_{ij} = \beta^Y_{ij}$, so
  $\gamma_{ij} = \mathrm{id}_{\mathbb{Z}_2}$ works.
  \item For $\{i,j\} = \{3,4\}$ we have $\beta^X_{34}(0) = 2 = \beta^Y_{34}(1)$
  and $\beta^X_{34}(1) = 3 = \beta^Y_{34}(0)$, so the swap
  $\gamma_{34}(k) = k + 1$ works.
\end{itemize}
Then the complexes are isomorphic. Moreover, both unions have $8$ branches. However, the two unions are not outer bi-Lipschitz equivalent, since in $X$ the contact orders $3$ happen in two different directions (both the axes), while in $Y$ it only occurs in the $x$-axis. 

\printbibliography

\end{document}